\documentclass[a4paper,10pt]{article}
\usepackage[latin1]{inputenc}
\usepackage[T1]{fontenc}
\usepackage{amsmath}
\usepackage{amsfonts}
\usepackage{amstext}
\usepackage{amsthm}
\usepackage[all]{xy}
\usepackage{geometry}
\usepackage{srcltx}
\usepackage{graphics}
\usepackage{epsfig}
\usepackage{subfigure}
\usepackage{slashed}
\usepackage{color}
\usepackage{amssymb}
\usepackage{srcltx}
\usepackage{comment}
\newtheorem{theorem}{Theorem}[section]
\newtheorem{lemma}[theorem]{Lemma}

\newtheorem{prop}[theorem]{Proposition}
\newtheorem{cor}[theorem]{Corollary}

\DeclareMathOperator{\im}{im}

\DeclareMathOperator{\sfl}{sf}
\DeclareMathOperator{\diag}{diag}

\DeclareMathOperator{\sgn}{sgn}
\DeclareMathOperator{\interior}{int}

\DeclareMathOperator{\GL}{GL}

\DeclareMathOperator{\inv}{inv}

\title{The Equivariant Spectral Flow and Bifurcation for Functionals with Symmetries - Part I}
 
\author{Marek Izydorek, Joanna Janczewska, Maciej Starostka and Nils Waterstraat}

\begin{document}
\date{}
\maketitle


\begin{abstract}
\noindent
We consider bifurcation of critical points from a trivial branch for families of functionals that are invariant under the orthogonal action of a compact Lie group. Based on a recent construction of an equivariant spectral flow by the authors, we obtain a bifurcation theorem that generalises well-established results of Smoller and Wasserman as well as Fitzpatrick, Pejsachowicz and Recht. Finally, we discuss some elementary examples of strongly indefinite systems of PDEs and Hamiltonian systems where the mentioned classical approaches fail but an invariance under an orthogonal action of $\mathbb{Z}_2$ makes our methods applicable and yields the existence of bifurcation.   
\end{abstract}

\section{Introduction}
Let $H$ be a real separable Hilbert space of infinite dimension and $f:I\times H\rightarrow\mathbb{R}$ a continuous map such that each $f_\lambda:=f(\lambda,\cdot):H\rightarrow\mathbb{R}$ is $C^2$ with derivatives depending continuously on the parameter $\lambda\in I:=[0,1]$. Let $0$ be a critical point of all $f_\lambda$ and consider the family of equations

\[(\nabla f_\lambda)(u)=0,\]
which now has $u=0$ as solution for all $\lambda\in I$. A parameter value $\lambda_0\in I$ is called a bifurcation point (of critical points) if in every neighbourhood $U\subset I\times H$ of $(\lambda_0,0)$ there is some $(\lambda,u)$ such that $(\nabla f_\lambda)(u)=0$ and $u\neq 0$. The existence of bifurcation points is a classical problem in nonlinear analysis that has been systematically studied for many decades. A central role is played by the second derivative $D^2_0f_\lambda$ at the critical point $0\in H$, which is a symmetric bounded bilinear form on $H$. By the Riesz-representation theorem it uniquely determines a selfadjoint operator $L_\lambda$ on $H$ such that

\begin{align}\label{introduction-Hessian}
\langle L_\lambda u,v\rangle_H=(D^2_0f_\lambda)(u,v),\quad u,v\in H,
\end{align}
which is called the Hessian of $f$ at $0\in H$. Note that it is an immediate consequence of the implicit function theorem that $L_\lambda$ is not invertible if $\lambda$ is a bifurcation point. Krasnoselskii considered in the sixties the case that $L_\lambda=I_H-\lambda K$, where $I_H$ denotes the identity on $H$ and $K$ is a compact selfadjoint operator. He showed in a celebrated theorem that the bifurcation points of $f$ are exactly those parameter values for which $L_\lambda$ has a non-trivial kernel or, in other words, $\frac{1}{\lambda}$ is an element of the spectrum $\sigma(K)$ of the compact operator $K$. More generally, let us now assume that the selfadjoint operators $L_\lambda$ are Fredholm, i.e., they have a finite dimensional kernel and a closed range. The following generalisation of Krasnoselskii's work is nowadays a common result in nonlinear analysis that can be found e.g. in the monographs \cite{Kielhoefer}, \cite{Mawhin}. Henceforth, we denote by $\mu_-(S)$ the Morse index of a selfadjoint Fredholm operator $S$, i.e., the number of negative eigenvalues of $S$ counted with multiplicities.

\begin{theorem}\label{thm-Morsejump}
If $\mu_-(L_\lambda)<\infty$ for all $\lambda\in I$, $L_0$, $L_1$ are invertible and

\begin{align}\label{Morsejump}
\mu_-(L_0)\neq\mu_-(L_1),
\end{align}
then there is a bifurcation point of critical points of $f$ in $(0,1)$.
\end{theorem}  
\noindent
It is readily seen that the invertibility of $L_0$ and $L_1$ cannot be lifted in this theorem (cf. \cite[\S 12.2]{book}). However, there is ample motivation to relax the assumption on the finiteness of the Morse indices and condition \eqref{Morsejump}. Firstly, sometimes symmetries of the functionals affect the applicability of \eqref{Morsejump}. For example, if the spectra of the operators $L_\lambda$ are symmetric in some neighbourhood about $0$, then the Morse indices are necessarily constant and the above theorem cannot be applied. Secondly, the finitness of the Morse indices excludes various important applications. For example, when studying solutions of Hamiltonian systems or non-cooperative elliptic systems of PDE as critical points of a suitable functional, the appearing operators $L_\lambda$ will usually not meet this condition.\\
The obstacle caused by constant Morse-indices was treated by Smoller and Wasserman in their seminal work \cite{SmollerWasserman} as follows. Assume that $G$ is a compact Lie group that acts orthogonally on $H$ and that each functional $f_\lambda$ is invariant under the action of $G$, i.e., $f_\lambda(gu)=f_\lambda(u)$ for all $g\in G$ and $u\in H$. Then the Hessians $L_\lambda$ are readily seen to be $G$-equivariant, i.e., $L_\lambda(gu)=gL_\lambda u$. If now $\mu_-(L_\lambda)<\infty$, then the direct sum $E^-(L_\lambda)$ of all eigenspaces with respect to negative eigenvalues is of finite dimension. As $E^-(L_\lambda)$ is easily seen to be invariant under the action of $G$, this space is a finite dimensional representation of the Lie group $G$. The terminology of a nice Lie group, that was introduced in \cite{SmollerWasserman}, will be recalled below in Section \ref{sect-main}.     

\begin{theorem}\label{thm-Morsenojump}
Assume that $G$ is nice, $L_0, L_1$ are invertible and $\mu_-(L_\lambda)<\infty$ for all $\lambda\in I$. If

\begin{align}\label{Morsenojump}
E^-(L_0)\ncong E^-(L_1),
\end{align}
where $\cong$ stands for isomorphic representations of $G$, then there is a bifurcation of critical points for $f_\lambda$ in $(0,1)$.
\end{theorem}
\noindent
Note that Theorem \ref{thm-Morsejump} follows from Theorem \ref{thm-Morsenojump} as isomorphic representations are of the same dimension and $\mu_-(L_\lambda)=\dim(E^-(L_\lambda))$. Smoller and Wasserman applied Theorem \ref{thm-Morsenojump} in \cite{SmollerWasserman} to study bifurcation of radial solutions of semilinear elliptic equations.\\
There have been various attempts to generalise Theorem \ref{thm-Morsejump} to the case when $\mu_-(L_\lambda)=\infty$. These were mostly tailored to specific applications like, e.g., bifurcation of branches of periodic solutions of Hamiltonian systems (cf. e.g. \cite{KryszewskiSzulkin}, \cite{Szulkin}). A very general approach to this problem was introduced by Fitzpatrick, Pejsachowicz and Recht in \cite{SFLPejsachowiczI}. The spectral flow is an integer-valued homotopy invariant that is defined for any path $L=\{L_\lambda\}_{\lambda\in I}$ of selfadjoint Fredholm operators that was introduced by Atiyah, Patodi and Singer in \cite{AtiyahPatodi} in connection with the Atiyah-Singer Index Theorem. Roughly speaking, the spectral flow counts the net number of eigenvalues crossing $0$ whilst the parameter $\lambda$ of the path traverses the interval. We recall the construction of the spectral flow below in Section \ref{sect-gequivsfl} and now just mention that if $\mu_-(L_\lambda)<\infty$ for all $\lambda\in I$, then 

\[\sfl(L)=\mu_-(L_0)-\mu_-(L_1).\]
Thus the following main theorem of \cite{SFLPejsachowiczI} is a natural generalisation of Theorem \ref{thm-Morsejump} which is applicable to any family of functionals $f_\lambda$ such that the associated Hessians $L_\lambda$ in \eqref{introduction-Hessian} are Fredholm operators.

\begin{theorem}\label{thm-FPR}
If $L_0$, $L_1$ are invertible and

\begin{align}\label{sflnontriv}
\sfl(L)\neq 0\in\mathbb{Z},
\end{align}
then there is a bifurcation point of critical points of $f$ in $(0,1)$.
\end{theorem} 
\noindent
Let us note that there are various efficient methods to compute the spectral flow, e.g., dimension reductions or crossing forms, that were in particular developed for applications in symplectic geometry (cf. e.g. \cite{SFLPejsachowiczII}, \cite{Robbin-Salamon}). These have yielded several bifurcation theorems for various types of differential equations, e.g., \cite{SFLPejsachowiczII,AleIchIndef,Homoclinics,Edinburgh}.\\
A natural question about Theorem \ref{thm-FPR} is if \eqref{sflnontriv} can be further relaxed. For example, if the spectra of the operators $L_\lambda$ are symmetric about $0$ by some symmetry of the functionals $f_\lambda$, then necessarily $\sfl(L)$ vanishes as net number of eigenvalues crossing through $0$. But if in this case there are pairs of eigenvalues crossing the axis in opposite direction, one might still have the idea that there should be a bifurcation. Interestingly, this is not the case by the following theorem of Alexander and Fitzpatrick from \cite{AFitzpatrick}.    

\begin{theorem}\label{bif-thm-AFi}
Let $L=\{L_\lambda\}_{\lambda\in I}$ be a path of selfadjoint Fredholm operators and $\lambda_0\in (0,1)$ such that $L_\lambda$ is invertible for $\lambda\neq \lambda_0$. If $\sfl(L)= 0$, then there exist an open interval $J\subset [0,1]$ containing $\lambda_0$, an open ball $B\subset H$ and a continuous family $f:J\times B\rightarrow\mathbb{R}$ of $C^2$-functionals such that $L_\lambda$ are the Hessians of $f_\lambda$ at $0\in H$ and $\nabla f_\lambda(0)=0$ holds for $\lambda\in J$, but there is no bifurcation of critical points for $f$ in $J$.
\end{theorem}
\noindent
The previous theorem suggests that Theorem \ref{thm-FPR} is optimal and we will show in Section \ref{sect-simpEx} that the predicted phenomenon is not at all pathologic. Indeed there are natural examples of differential equations to which Theorem \ref{bif-thm-AFi} applies.\\
The aim of this work is to introduce a result that generalises all previously mentioned theorems about the existence of bifurcation and which shows that despite of Theorem \ref{bif-thm-AFi} there is still bifurcation under a suitable symmetry assumption on the functionals $f_\lambda$. The authors recently introduced in \cite{MJN21} the $G$-equivariant spectral flow $\sfl_G(L)$ for paths of selfadjoint Fredholm operators $L=\{L_\lambda\}_{\lambda\in I}$ that are equivariant under the orthogonal action of a compact Lie group, i.e., $L_\lambda(gu)=g(L_\lambda u)$ for all $u\in H$ and $g\in G$. This novel homotopy invariant is an element of the representation ring $RO(G)$ that was introduced by Segal in \cite{Segal}. It was shown in \cite{MJN21} that 

\[F(\sfl_G(L))=\sfl(L)\in\mathbb{Z}\]
by a natural homomorphism $F:RO(G)\rightarrow\mathbb{Z}$ that we recall below in Section \ref{sect-gequivsfl}. Thus even if the spectral flow $\sfl(L)\in\mathbb{Z}$ vanishes, $\sfl_G(L)$ can be non-trivial in $RO(G)$.  Moreover, if $\mu_-(L_\lambda)<\infty$, then 

\begin{align}\label{equ-GequivMorse}
\sfl_G(L)=[E^-(L_0)]-[E^-(L_1)]\in RO(G),
\end{align}
where the square brackets stand for isomorphism classes of representations of the compact Lie group $G$. Consequently, if in this case $\sfl_G(H)\neq 0\in RO(G)$, then $E^-(L_0)$ and $E^-(L_1)$ are non-isomorphic $G$-representations and thus \eqref{Morsenojump} holds. The main theorem of this work states that if $L_0$, $L_1$ are invertible and $\sfl_G(L)$ is non-trivial in the representation ring $RO(G)$ of the compact Lie group $G$, then there is a bifurcation of critical points of $f$. As already said, this is a strong statement as it implies Theorem \ref{thm-Morsenojump} and Theorem \ref{thm-FPR} (and thus ultimately also Theorem \ref{thm-Morsejump}). The price to pay is that it is a rather challenging task to join the proofs of Theorem \ref{thm-Morsenojump} and Theorem \ref{thm-FPR} and the majority of this paper is devoted to this issue. On the other hand, it turns out that $\sfl_G(L)$ is quite applicable. Indeed, major tools for computing the classical spectral flow $\sfl(L)$ in $\mathbb{Z}$ carry over to the $G$-equivariant case in $RO(G)$ and pave the way to various applications to Hamiltonian systems and PDEs that are invariant under actions of subgroups of the matrix groups $O(n)$ or $SO(n)$. These methods to compute $\sfl_G(L)$ and applications of them will be introduced in a second part of this work. Here we restrict to the simplest non-trivial group $\mathbb{Z}_2$ and consider even functionals $f_\lambda$ as in \cite{Clapp}. In our first application we study Dirichlet boundary value problems for systems of non-linear PDEs. We construct a natural class of examples where Theorem \ref{thm-FPR} fails but it is readily seen that $\sfl_G(L)$ is non-trivial in $RO(G)$, and thus there is a bifurcation that would not have been found without taking the group action into account. As second class of examples, we consider bifurcation of homoclinic solutions of Hamiltonian systems. There have been many studies devoted to bifurcation for functionals that are invariant under orthogonal actions of compact Lie groups by degree theory (see e.g. \cite{shortdegree} and ref. therein). As those methods only apply to functionals where the Hessians $L_\lambda$ are compact perturbations of a fixed operator, homoclinic solutions are out of their scope. We study even Hamiltonian systems and construct an example for which we can show bifurcation of homoclinic solutions by the main theorem of this work. Moreover, as the classical spectral flow $\sfl(L)$ vanishes in our example, also Theorem \ref{thm-FPR} is not applicable. This shall emphasize the strength of our findings and should be an appetizer for the study of continuous groups in an upcoming second part.\\
The paper is structured as follows. In the next two sections we recap the classical spectral flow, introduce the $G$-equivariant spectral flow from \cite{MJN21} and state our main theorem. Section \ref{sect-cogredient} provides a first pillar of the proof of our main theorem. We show the existence of a $G$-equivariant parametrix for any path of $G$-equivariant selfadjoint Fredholm operators, which allows to reduce a path of Hessians to a normal form. In Section \ref{sect-thm:main} we use the result of Theorem  \ref{thm-cogredpar} to prove Theorem \ref{thm:main}, which we do in several steps. In the final section of this paper we discuss the announced examples of even functionals that should stress the high applicability of our work.



\section{The $G$-equivariant Spectral Flow}\label{sect-gequivsfl}
The aim of this section is to introduce the $G$-equivariant spectral flow, where we follow the authors recent work \cite{MJN21}. Let us first recap the definition of the classical spectral flow.\\ 
Let $H$ be a real separable Hilbert space and let $\mathcal{FS}(H)$ be the set of all selfadjoint Fredholm operators on $H$ with the norm topology. It was shown by Atiyah and Singer in \cite{AtiyahSinger} that $\mathcal{FS}(H)$ has three connected components

\begin{align*}
\mathcal{FS}^+(H)&:=\{T\in\mathcal{FS}(H):\, \sigma_{ess}(T)\subset(0,+\infty)\}\\
\mathcal{FS}^-(H)&:=\{T\in\mathcal{FS}(H):\, \sigma_{ess}(T)\subset(-\infty,0)\},
\end{align*}

and
\begin{align*}
\mathcal{FS}^i(H):=\mathcal{FS}(H)\setminus\mathcal{FS}^\pm(H),
\end{align*}
where $\sigma_{ess}(T)$ denotes the essential spectrum, i.e., the set of all $\lambda\in\mathbb{R}$ such that $\lambda-T$ is not a Fredholm operator. The operators in $\mathcal{FS}^+(H)$ have a finite Morse index

\begin{align}\label{Morse}
\mu_-(T)=\dim\left(\oplus_{\mu<0}\{u\in H:\, T u=\mu u\}\right),
\end{align}
i.e., they have at most finitely many negative eigenvalues including multiplicities. In general, for every $T\in\mathcal{FS}(H)$, there is a neighbourhood of $0$ that either belongs to the resolvent set or it contains finitely many eigenvalues including multiplicities (cf. \cite{ProcHan}, \cite{Fredholm}).\\
Let now $L=\{L_\lambda\}_{\lambda\in I}$ be a path in $\mathcal{FS}(H)$. As the spectra of the operators $L_\lambda$ cannot accumulate at $0$, it can be shown that there is a partition $0=\lambda_0<\ldots< \lambda_N=1$ of the unit interval and $a_i>0$, $i=1,\ldots N$, such that $[\lambda_{i-1},\lambda_i]\ni \lambda\mapsto\chi_{[-a_i,a_i]}(L_\lambda)\in\mathcal{L}(H)$ are continuous families of finite rank projections, where $\chi_{[a,b]}(T)$ denotes the spectral projection of a selfadjoint operator $T$ with respect to the interval $[a,b]\subset\mathbb{R}$. Then, for $i=1,\ldots,N$, the spaces $E(L_\lambda,[0,a_i]):=\im(\chi_{[0,a_i]}(L_\lambda))$, $\lambda_{i-1}\leq \lambda\leq \lambda_i$, are finite dimensional and the spectral flow of the path $L$ was defined by Phillips in \cite{Phillips} by

\begin{align}\label{sfl}
\sfl(L)=\sum^N_{i=1}{(\dim(E(L_{\lambda_i},[0,a_i]))-\dim(E(L_{\lambda_{i-1}},[0,a_i])))}\in\mathbb{Z}.
\end{align}  
Note that $E(L_\lambda,[0,a])$ is the direct sum of the eigenspaces of $L_\lambda$ for eigenvalues in the interval $[0,a]$.\\ 
Let now $G$ be a compact Lie group. A (real) representation of $G$ is a pair $(V,\rho)$ consisting of a finite dimensional (real) vector space $V$ and a group homomorphism $\rho:G\rightarrow\GL(V)$. Two representations $(V_1,\rho_1)$, $(V_2,\rho_2)$ of $G$ are isomorphic if there is an isomorphism $\alpha:V_1\rightarrow V_2$ that is $G$-equivariant, i.e., $\rho_2(g)\circ \alpha=\alpha\circ\rho_1(g)$ for all $g\in G$. Two representations of $G$ can be added by the direct sum and this turns the set of isomorphism classes of representations of $G$ into a commutative monoid. The elements of the associated Grothendieck group $RO(G)$ are formal differences $[U]-[V]$ of isomorphism classes of $G$-representations modulo the equivalence relation generated by $[U]-[V]\sim [U\oplus W]-[V\oplus W]$. The neutral element in $RO(G)$ is $[V]-[V]$ for any $G$-representation $V$, and the inverse element of $[U]-[V]$ is $[V]-[U]$. $RO(G)$ was introduced by Segal in \cite{Segal}, who called it the representation ring of $G$ as the tensor product of representations actually yields a ring structure on this Grothendieck group. We follow this terminology even though we will never use the ring structure of $RO(G)$ and consider it merely as an abelian group.\\
Let now $G$ be a compact Lie group that acts orthogonally on $H$. We denote by $\mathcal{FS}(H)^G$ the set of $G$-equivariant selfadjoint Fredholm operators, i.e.,

\[T(gu)=g(T u),\qquad u\in H,\quad g\in G,\]   
and by $\mathcal{FS}^\pm(H)^G$ and $\mathcal{FS}^i(H)^G$ the corresponding subsets of the connected components of $\mathcal{FS}(H)$.\\
Let now $L=\{L_\lambda\}_{\lambda\in I}$ be a path in $\mathcal{FS}(H)^G$. As the operators $L_\lambda$ are $G$-equivariant, it follows that the spaces $E(L_{\lambda},[0,a])$ in \eqref{sfl} are $G$-invariant. Thus they define equivalence classes of $G$-representations and consequently the idea of \eqref{sfl} carries over to $RO(G)$ by setting   

\begin{align}\label{sfl-equiv}
\sfl_G(L)=\sum^N_{i=1}{([E(L_{\lambda_i},[0,a_i])]-[E(L_{\lambda_{i-1}},[0,a_i])])}\in RO(G).
\end{align} 
Phillips proved in \cite{Phillips} that \eqref{sfl} only depends on the path $L$ and not on the choices of the partition $0=\lambda_0<\ldots< \lambda_N=1$ of the unit interval and the numbers $a_i>0$. Recently, the authors showed in \cite{MJN21} that the same is true for \eqref{sfl-equiv} in $RO(G)$, and thus this equivariant spectral flow is well defined. Note that if $G$ is trivial, then representations are isomorphic if and only if they are of the same dimension. Hence $RO(G)\cong\mathbb{Z}$ in this case and \eqref{sfl-equiv} can be identified with the ordinary spectral flow \eqref{sfl}. In general, there is a canonical homomorphism 

\[F:RO(G)\rightarrow\mathbb{Z},\quad [U]-[V]\mapsto\dim(U)-\dim(V),\]
and it follows from \eqref{sfl} and \eqref{sfl-equiv} that

\begin{align}\label{forgetfull}
F(\sfl_G(L))=\sfl(L).
\end{align}
Consequently, the classical spectral flow of $L$ has to vanish if $\sfl_G(L)$ is trivial. On the other hand, in \cite{MJN21} there is a simple example of a path of $G=\mathbb{Z}_2$-equivariant operators such that $\sfl_G(L)\in RO(G)\cong\mathbb{Z}\oplus\mathbb{Z}$ is non-trivial even though $\sfl(L)=0\in\mathbb{Z}$.\\
Finally, it was shown in \cite{MJN21} that all basic properties of the spectral flow hold mutatis mutandis for its $G$-equivariant generalisation \eqref{sfl-equiv}, e.g.,

\begin{itemize}
 \item[(i)] If $L_\lambda\in\GL(H)\cap\mathcal{FS}(H)^G$ for all $\lambda\in I$, then 
 
 \begin{align}\label{sflNormalisation}
 \sfl_G(L)=0\in RO(G).
 \end{align}
 \item[(ii)] Let $H=H_1\oplus H_2$, where $H_1$, $H_2$ are $G$-invariant and such that $L_\lambda\mid_{H_i}\in\mathcal{FS}(H_i)^G$ for $i=1,2$,  $\lambda\in I$. Then

\begin{align}\label{sflAdditivity}
\sfl_G(L)=\sfl_G(L\mid_{H_1})+\sfl_G(L\mid_{H_2})\in RO(G).
\end{align}
\item[(iii)] If $L,L'$ are two paths in $\mathcal{FS}(H)^G$ such that $L_1=L'_0$, then

\begin{align}\label{concatenation}
\sfl_G(L\ast L')=\sfl_G(L)+\sfl_G(L')\in RO(G),
\end{align} 
where $L\ast L'$ denotes the concatenation of $L'$ and $L$.
\item[(iv)] If the path $L^-$ is defined by $L^-_\lambda=L_{1-\lambda}$, $\lambda\in I$, then

\begin{align}\label{reversal}
\sfl_G(L^-)=-\sfl_G(L)\in RO(G).
\end{align}
\item[(v)] If $h:I\times I\rightarrow\mathcal{FS}(H)^G$ is a homotopy such that $h(s,0)$ and $h(s,1)$ are invertible for all $s\in I$, then

\begin{align}\label{sflHomotopy}
\sfl_G(h(0,\cdot))=\sfl_G(h(1,\cdot))\in RO(G).
\end{align}
\end{itemize}
Moreover, let us recall the following proposition from \cite[Prop. 3.2]{MJN21}, where $E^-(T)$ denotes the direct sum of the eigenspaces with respect to negative eigenvalues of an operator $T\in\mathcal{FS}^+(H)$.

\begin{prop}\label{prop-sfldiffMorse}
If $L=\{L_\lambda\}_{\lambda\in I}$ is a path in $\mathcal{FS}^+(H)^G$, then

\[\sfl_G(L)=[E^-(L_0)]-[E^-(L_1)]\in RO(G).\]
\end{prop}
\noindent 
Finally, let us note that further properties of the $G$-equivariant spectral flow, and in particular a generalisation of \eqref{sflHomotopy} to homotopies with non-invertible endpoints, can be found in \cite[\S 2.3]{MJN21}. The latter is necessary for the following proposition about compact perturbations in $\mathcal{FS}(H)^G$.

\begin{prop}\label{prop-FixedPerturbation}
Let $L=\{L_\lambda\}_{\lambda\in I}$ and $L'=\{L'_\lambda\}_{\lambda\in I}$ be paths in $\mathcal{FS}(H)^G$ such that $L_0=L'_0$, $L_1=L'_1$ and $L_\lambda-L'_\lambda$ is compact for all $\lambda\in I$. Then

\[\sfl_G(L)=\sfl_G(L').\] 
\end{prop}

\begin{proof}
We set $K_\lambda=L_\lambda-L'_\lambda$ and first note that by assumption $K_0=K_1=0$. Thus

\[h:I\times I\rightarrow\mathcal{FS}(H)^G,\qquad h(s,\lambda)=L'_\lambda+sK_\lambda\]
is a homotopy with fixed endpoints such that $h(0,\lambda)=L'_\lambda$ and $h(1,\lambda)=L_\lambda$. This shows the claimed equality as the spectral flow is invariant under homotopies with fixed endpoints by \cite[Cor. 2.9]{MJN21}.   
\end{proof}
\noindent
We note that the previous proposition in particular applies to the case that $L_\lambda=A+K_\lambda$ and $L'_\lambda=A+K_0$, $\lambda\in I$, for some fixed $A\in\mathcal{FS}(H)^G$ and a closed path $\{K_\lambda\}_{\lambda\in I}$ of compact selfadjoint $G$-equivariant operators. Then $\sfl_G(L)=\sfl_G(L')=0\in RO(G)$ as it directly follows from the definition \eqref{sfl-equiv} that the spectral flow of a constant path vanishes.


\section{Main Theorem and Corollaries}\label{sect-main}
We consider equations of the type $\nabla f_\lambda(u)=0$, where $f:I\times H\rightarrow \mathbb{R}$ is a family of $C^2$-functionals on an infinite dimensional real separable Hilbert space $H$, and we assume that $\nabla f_\lambda(0)=0$ for all $\lambda\in I$, i.e. $0\in H$ is a critical point of all functionals $f_\lambda$. A \textit{bifurcation point} is a parameter value $\lambda^\ast\in I$ at which non-trivial critical points branch off from the trivial ones $I\times\{0\}$, i.e., in every neighbourhood of $(\lambda^\ast,0)\in I\times H$ there is some $(\lambda,u)$ such that $\nabla f_\lambda(u)=0$ and $u\neq 0$. A crucial role for studying the existence of bifurcation points is played by the family of Hessians $L_\lambda$, which are bounded selfadjoint operators on $H$ that are induced by the second derivatives $D^2_0f_\lambda$ of $f_\lambda$ at $0\in H$ as in \eqref{introduction-Hessian}. It is a common assumption that the operators $L_\lambda$ are Fredholm, i.e., they are elements of the space $\mathcal{FS}(H)$ introduced in the previous section. Let us now assume in addition that each $f_\lambda$ is $G$-invariant, i.e., $f(g u)=f(u)$, $g\in G$, where $G$ is a compact Lie group acting orthogonally on $H$. Under this assumption, the operators $L_\lambda$ are $G$-equivariant, i.e. $L_\lambda(gu)=g(L_\lambda u)$ by \cite[\S 1.3]{IzeVignoli}. Thus the $G$-equivariant spectral flow $\sfl_G(L)$ is defined.\\
Henceforth, we assume that $G$ is \textit{nice} in the sense of Smoller and Wasserman's work \cite{SmollerWasserman}, i.e., any orthogonal representations $E$ and $F$ of $G$ are isomorphic if the quotients $D(E)/S(E)$ and $D(F)/S(F)$ of the unit discs by the unit spheres have the same $G$-equivariant homotopy type. Let us stress that, e.g., if $G_0$ denotes the connected component of the identity in $G$, then $G$ is nice if $G/G_0$ is trivial or a finite product of $\mathbb{Z}_2$ or $\mathbb{Z}_3$. Thus, in particular, $S^1$, $O(n)$ and $SO(n)$ are nice. The following theorem is the main result of this work. 

\begin{theorem}\label{thm:main}
If $L_\lambda\in\mathcal{FS}(H)^G$, $\lambda\in I$, $L_0$,$L_1$ are invertible and 

\[\sfl_G(L)\neq0\in RO(G),\]
then there is a bifurcation of critical points for $f$. 
\end{theorem}
\noindent
The following corollary is an immediate consequence of Proposition \ref{prop-sfldiffMorse} and Theorem \ref{thm:main}. It is the main result of Smoller and Wasserman's work \cite{SmollerWasserman} that we stated in the introduction in Theorem \ref{thm-Morsenojump}.

\begin{cor}\label{cor:main}
If $L_\lambda\in\mathcal{FS}^+(H)^G$, $\lambda\in I$, $L_0$,$L_1$ are invertible and $E^-(L_0)$ and $E^-(L_1)$ are not isomorphic as $G$-representations, then there is a bifurcation of critical points for $f$. 
\end{cor}
\noindent
By \eqref{forgetfull}, we also reobtain the main theorem of Fitzpatrick, Pejsachowicz and Recht's work \cite{SFLPejsachowiczI} that we stated in the introduction in Theorem \ref{thm-FPR}.

\begin{cor}\label{cor:sfl}
If $L_\lambda\in\mathcal{FS}(H)$, $\lambda\in I$, $L_0$,$L_1$ are invertible and $\sfl(L)\neq0\in\mathbb{Z}$, then there is a bifurcation of critical points for $f$. 
\end{cor}
\noindent
Finally, let us emphasize that Theorem \ref{thm-Morsejump} and Krasnoselskii's bifurcation theorem (of which we reminded above Theorem \ref{thm-Morsejump}) are immediate consequences of both Corollary \ref{cor:main} and Corollary \ref{cor:sfl}. Consequently, they ultimately are also covered by our Theorem \ref{thm:main}.


\section{The $G$-equivariant Cogredient Parametrix}\label{sect-cogredient}
We call an operator $Q\in\mathcal{L}(H)$ a symmetry if it is of the form $Q=P-(I_H-P)=2P-I_H$ for an orthogonal projection $P$ which has infinite dimensional kernel and range. Note that any symmetry $Q$ satisfies $Q^2=I_H$ and $Q\in\mathcal{FS}^i(H)$. Moreover, $Q$ is $G$-equivariant if and only if $\im(P)$ and $\ker(P)$ are $G$-invariant subspaces of $H$.\\
The aim of this section is the proof of the following theorem, which is a pillar of the proof of Theorem \ref{thm:main}. Henceforth we denote by $\mathcal{KS}(H)$ the space of all selfadjoint compact operators with the norm topology, and by $\mathcal{KS}(H)^G$ its $G$-equivariant elements. Similarly, $\GL(H)^G$ stands for the $G$-equivariant invertible operators. 

\begin{theorem}\label{thm-cogredpar}
Let $L=\{L_\lambda\}_{\lambda\in I}$ be a path in $\mathcal{FS}^i(H)^G$. Then there is a $G$-equivariant symmetry $Q\in\mathcal{FS}^i(H)^G$ and paths $M=\{M_\lambda\}_{\lambda\in I}$ in $\GL(H)^G$ and $K=\{K_\lambda\}_{\lambda\in H}\in\mathcal{KS}(H)^G$ such that

\[M^\ast_\lambda L_\lambda M_\lambda=Q+K_\lambda,\qquad\lambda\in I.\] 
\end{theorem}
\noindent
The remainder of this section is devoted to the proof of this theorem. Let us first sketch the idea. We consider for a fixed $G$-equivariant symmetry $Q\in\mathcal{FS}^i(H)^G$ the map

\begin{align}\label{defQ}
\pi_Q:\GL(H)^G\times\mathcal{KS}(H)^G\rightarrow\mathcal{FS}^i(H)^G,\quad \pi_Q(M,K)=MQM^\ast+K.
\end{align}
Note that $\im(\pi_Q)$ is indeed in $\mathcal{FS}^i(H)^G$ as $G$ acts orthogonally and thus the adjoint $M^\ast$ is $G$-equivariant as well. Clearly Theorem \ref{thm-cogredpar} is shown if we can prove that for some $Q$, the path $L$ can be lifted to $\GL(H)^G\times\mathcal{KS}(H)^G$, i.e., if there is a continuous map $\tilde{L}:I\rightarrow \GL(H)^G\times\mathcal{KS}(H)^G$ such that $L_\lambda=\pi_Q\circ\tilde{L}$ for all $\lambda\in I$. In the non-equivariant case, this was done in \cite{SFLPejsachowiczI} by showing that \eqref{defQ} is the projection of a fibre bundle. Then the desired lifting of $L$ can be obtained from a global section of the pullback bundle of \eqref{defQ} by $L$, which exists as the latter bundle has a contractible base space.\\
Unfortunately, in our more general setting, $\pi_Q$ is not necessarily surjective, which affects the argument of \cite{SFLPejsachowiczI}. Indeed, we will see below that in general $\im(\pi_Q)$ is a union of connected components of $\mathcal{FS}^i(H)^G$ if $G$ is non-trivial. Moreover, at the end of this section we provide an example where these components are not all of $\mathcal{FS}^i(H)^G$.\\
Before we begin the proof of Theorem \ref{thm-cogredpar}, we note the following simple lemma about functional calculus that will be used throughout the rest of the paper. 

\begin{lemma}\label{FunctCalc}
Let $T\in\mathcal{L}(H)$ be selfadjoint and $f:\sigma(T)\rightarrow\mathbb{R}$ a continuous function on the spectrum of $T$. If $T$ is $G$-equivariant, then so is $f(T)$. 
\end{lemma}

\begin{proof}
Note that $p(T)$ is $G$-equivariant for every polynomial, and for every $\varepsilon>0$ there is a polynomial such that $\|f-p\|_\infty<\varepsilon$ on $\sigma(T)$. Hence it follows for $g\in G$ that

\begin{align*}
\|f(T)g-gf(T)\|\leq\|f(T)g-p(T)g\|+\|gp(T)-gf(T)\|\leq 2\|p(T)-f(T)\|\leq 2\|p-f\|_\infty\leq 2\varepsilon,
\end{align*}
which implies that $f(T)g=gf(T)$.
\end{proof}
\noindent
The following quite technical proposition shows the existence of a local section of \eqref{defQ}. Let us point out that even in the case of a trivial group action, the result is more general than the corresponding Lemma 2.2 in \cite{SFLPejsachowiczI}. The latter only constructs a section in a neighbourhood of any symmetry in $\mathcal{FS}^i(H)^G$ which is not enough for our purposes due to the already mentioned lack of surjectivity of $\pi_Q$. 

\begin{prop}\label{prop-CogrPar}
For any $S\in\mathcal{FS}^i(H)^G$ there is a $G$-equivariant symmetry $Q_S$, an open neighbourhood $\mathcal{U}_S$ of $S$ in $\mathcal{FS}^i(H)^G$ and a map $\sigma_S:\mathcal{U}_S\rightarrow\GL(H)^G\times\mathcal{KS}(H)^G$ such that

\[(\pi_{Q_S}\circ\sigma_S)(T)=T\quad\text{for all}\,\,\,\, T\in\mathcal{U}_S.\]
\end{prop}

\begin{proof}
Let $K$ be the orthogonal projection onto the kernel of $S$. Then $K\in\mathcal{KS}(H)^G$ as $\ker(S)$ is $G$-invariant and of finite dimension, and moreover 

\begin{align}\label{V}
V:=S+K\in GL(H)^G.
\end{align}
Henceforth we denote by $P_+(V)=\chi_{(0,\infty)}(V)$ and $P_-(V)=\chi_{(-\infty,0)}(V)$ the projections on the positive and negative spectral subspaces of $V$. Note that these operators are $G$-equivariant by Lemma \ref{FunctCalc}. We set

\begin{align}\label{QS}
Q_S=2P_+(V)-I_H\in\mathcal{FS}^i(H)^G
\end{align}
and choose a neighbourhood $\tilde{\mathcal{U}}\subset\mathcal{FS}^i(H)^G$ of $Q_S$ that consists of invertible operators. As above, we let $P_+(T)$ and $P_-(T)$ denote the orthogonal projections onto the positive and negative spectral subspaces for $T\in\tilde{\mathcal{U}}$ which are again $G$-equivariant. As $P_\pm(T)$ continuously depend on $T\in\tilde{\mathcal{U}}$ and as $\GL(H)^G$ is open, there is a neighbourhood $\mathcal{U}\subset\tilde{\mathcal{U}}$ of $Q_S$ such that

\[P_+(T) P_+(Q_S)+P_-(T) P_-(Q_S)\in GL(H)^G,\quad T\in \mathcal{U},\]  
where we use that this operator is the identity for $T=Q_S$. Consequently, for each $T\in\mathcal{U}$, the restriction of $P_\pm(T)$ to $\im(P_\pm(Q_S))$ is a bijection onto $\im(P_\pm(T))$. Henceforth, we set $H_\pm:=\im(P_\pm(Q_S))$ to simplify notation.\\
We now consider for $T\in\mathcal{U}$ the bilinear form on $H_+$ defined by

\[B(T)(u,v)=\langle TP_+(T)u,P_+(T)v\rangle,\quad u,v\in H_+.\]  
Then clearly $B(T)$ is bounded, symmetric and positive definite. Moreover, $B(T)$ is $G$-invariant as $T$ and $P_+(T)$ are $G$-equivariant and $G$ acts orthogonally. Let $\widetilde{T}$ be the Riesz-representation of $B(T)$, i.e., the unique selfadjoint operator such that

\[B(T)(u,v)=\langle\widetilde{T}u,v\rangle,\quad u,v\in H_+.\]
As $\widetilde{T}$ is unique and $B(T)$ is $G$-invariant, it is readily seen that $\widetilde{T}$ is $G$-equivariant. Thus, again by Lemma \ref{FunctCalc}, the inverse square-root $S_+(T):=\widetilde{T}^{-\frac{1}{2}}$ is $G$-equivariant as well. Moreover,

\begin{align*}
\begin{split}
\langle u,v\rangle&=\langle\widetilde{T}^{-1}\widetilde{T}u,v\rangle=\langle \widetilde{T}^{-\frac{1}{2}}\widetilde{T}u,\widetilde{T}^{-\frac{1}{2}}v\rangle=\langle \widetilde{T}\widetilde{T}^{-\frac{1}{2}}u,\widetilde{T}^{-\frac{1}{2}}v\rangle=B(T)(\widetilde{T}^{-\frac{1}{2}}u,\widetilde{T}^{-\frac{1}{2}}v)\\
&=\langle TP_+(T)\widetilde{T}^{-\frac{1}{2}}u,P_+(T)\widetilde{T}^{-\frac{1}{2}}v\rangle=\langle \widetilde{T}^{-\frac{1}{2}}P_+(T)TP_+(T)\widetilde{T}^{-\frac{1}{2}}u,v\rangle\\
&=\langle S_+(T)P_+(T)TP_+(T)S_+(T)u,v\rangle,\quad u,v\in H_+,
\end{split}
\end{align*}
which implies that 

\begin{align}\label{equ-S+}
S_+(T)P_+(T)TP_+(T)S_+(T)=I_{H_+}.
\end{align}
In the same way, we can construct a family $S_-:\mathcal{U}\rightarrow\GL(H_-)^G$ such that for all $T\in\mathcal{U}$

\begin{align*}
-\langle u,v\rangle=\langle S_-(T)P_-(T) TP_-(T)S_-(T)u,v\rangle,\quad u,v\in H_-,
\end{align*}
and thus

\begin{align}\label{equ-S-}
S_-(T)P_-(T) TP_-(T)S_-(T)=-I_{H_-}.
\end{align}
If we now define $S_0:\mathcal{U}\rightarrow\GL(H)^G$ by

\[S_0(T)=P_+(T)S_+(T)P_+(Q_S)-P_-(T)S_-(T)P_-(Q_S),\]
then it follows from \eqref{equ-S+} and \eqref{equ-S-} that


\begin{align*}
S_0(T)^\ast TS_0(T)=P_+(Q_S)-P_-(Q_S)=Q_S,\quad T\in\mathcal{U}.
\end{align*}
Consequently, the map 

\[\sigma:\mathcal{U}\rightarrow\GL(H)^G\times\mathcal{KS}(H)^G,\quad \sigma(T)=((S_0(T)^{-1})^\ast,0),\]
satisfies

\begin{align}\label{equ-picircsigma}
(\pi_{Q_S}\circ\sigma)(T)=\pi_{Q_S}((S_0(T)^{-1})^\ast,0)=(S_0(T)^{-1})^\ast Q_S S_0(T)^{-1}=T,\quad T\in\mathcal{U}.
\end{align}
Let us recall that $\mathcal{U}$ is a neighbourhood of the symmetry $Q_S$ that was defined in \eqref{QS}. Our next aim is to get a section $\sigma_V$ in a neighbourhood $\mathcal{U}_V$ of $V$ in \eqref{V}. In \cite{SFLPejsachowiczI} it was noted that $\mathcal{G}:=\GL(H)\times\mathcal{KS}(H)$ is a topological group with respect to 

\begin{align}\label{groupmult}
(M,K)\cdot (\tilde{M},\tilde{K})=(M\tilde{M}, K+M\tilde{K}M^\ast),
\end{align}
and there is an action $\tau$ of $\mathcal{G}$ on $\mathcal{FS}^i(H)$ defined by 

\[\tau_{h}(L)= MLM^\ast+K,\quad L\in\mathcal{FS}^i(H),\quad h=(M,K)\in \mathcal{G}.\]
Now $\GL(H)^G\times\mathcal{KS}(H)^G$ is a closed subgroup of $\mathcal{G}$ and the action $\tau$ restricts to an action of it on $\mathcal{FS}^i(H)^G$. We set $h:=(|V|^\frac{1}{2},0)\in \GL(H)^G\times\mathcal{KS}(H)^G$ and see by functional calculus that

\[\tau_h(Q_S)=|V|^\frac{1}{2}Q_S|V|^\frac{1}{2}=V,\] 
where \eqref{QS} is used as well as the obvious equality $\sqrt{|x|}(2\chi_{(0,\infty)}(x)-1)\sqrt{|x|}=x$, $x\in\mathbb{R}$. Thus $\mathcal{U}_V:=\tau_h(\mathcal{U})$ is an open neighbourhood of $V$ in $\mathcal{FS}^i(H)^G$. We set 

\[\sigma_V:\mathcal{U}_V\rightarrow \GL(H)^G\times\mathcal{KS}(H)^G,\quad \sigma_V(T)=h\cdot \sigma(\tau_{h^{-1}}(T)),\quad T\in\mathcal{U}_V,\]
where $h=(|V|^\frac{1}{2},0)\in \GL(H)^G\times\mathcal{KS}(H)^G$ as above and the group multiplication \eqref{groupmult} is used. To show that $\pi_{Q_S}\circ\sigma_V=id\mid_{\mathcal{U}_V}$, first note that for any $h_1,h_2\in \GL(H)^G\times\mathcal{KS}(H)^G$,

\[\pi_{Q_S}(h_1\cdot h_2)=\tau_{h_1}(\pi_{Q_S}(h_2)),\]
which directly follows from the definition of $\tau$ and \eqref{groupmult}. Thus by \eqref{equ-picircsigma}

\begin{align}\label{QSsigmaV}
\begin{split}
(\pi_{Q_S}\circ\sigma_V)(T)&=\pi_{Q_S}(h\cdot \sigma(\tau_{h^{-1}}(T)))=\tau_h(\pi_{Q_S}(\sigma(\tau_{h^{-1}}(T))))\\
&=\tau_h(\tau_{h^{-1}}(T))=T,\quad T\in\mathcal{U}_V,
\end{split}
\end{align}
as claimed. Finally, we set $\tilde{h}=(I_H,-K)\in \GL(H)^G\times\mathcal{KS}(H)^G$ and define $\mathcal{U}_S=\tau_{\tilde{h}}(\mathcal{U}_V)$ as well as

\[\sigma_S:\mathcal{U}_S\rightarrow \GL(H)^G\times\mathcal{KS}(H)^G,\quad \sigma_S(T)=\tilde{h}\cdot \sigma_V(\tau_{\tilde{h}^{-1}}(T)),\quad T\in\mathcal{U}_S.\]
Then $\mathcal{U}_S$ is an open neighbourhood of $S$ and the same computation as in \eqref{QSsigmaV} shows that indeed $\pi_{Q_S}\circ\sigma_S=id\mid_{\mathcal{U}_S}$, which ends our proof. 
\end{proof}
\noindent
The following lemma and its corollary concern the image of the map $\pi_Q$.

\begin{lemma}\label{lem-CogrPar}
Let $Q_1, Q_2$ be symmetries and $\pi_{Q_1}$, $\pi_{Q_2}$ the associated maps in \eqref{defQ}. Let $A,B$ be subsets of $\mathcal{FS}^i(H)^G$ such that $A\subset\im(\pi_{Q_1})$ and $B\subset\im(\pi_{Q_2})$. If $A\cap B\neq\emptyset$, then $A\cup B$ is contained in the image of $\pi_{Q_1}$ (and likewise of $\pi_{Q_2}$). 
\end{lemma}

\begin{proof}
If $S\in A\cap B$, then there are $(M_1,K_1), (M_2,K_2)\in\GL(H)^G\times\mathcal{KS}(H)^G$ such that

\[S=\pi_{Q_1}(M_1,K_1)=\pi_{Q_2}(M_2,K_2).\]
Thus

\[S=M_1Q_1M^\ast_1+K_1=M_2Q_2M^\ast_2+K_2,\]
which implies 

\[Q_2=M^{-1}_2M_1Q_1M^\ast_1(M^{-1}_2)^\ast+M^{-1}_2(K_1-K_2)(M^{-1}_2)^\ast.\]
If we now set $\tilde{h}=(M^{-1}_2M_1,M^{-1}_2(K_1-K_2)(M^{-1}_2)^\ast)\in \GL(H)^G\times\mathcal{KS}(H)^G$, then a direct computation yields

\[\pi_{Q_2}(h)=\pi_{Q_1}(h\cdot \tilde{h}),\quad h\in \GL(H)^G\times\mathcal{KS}(H)^G.\]
 Thus $B$ is contained in the image of $\pi_{Q_1}$.
\end{proof}
\noindent 

\begin{cor}\label{cor-CogrPar}
For every symmetry $Q$, the image of $\pi_Q$ is a union of connected components of $\mathcal{FS}^i(H)^G$. 
\end{cor}

\begin{proof}
We show that $\im(\pi_Q)$ is open and closed in $\mathcal{FS}^i(H)^G$. Firstly, if $S\in\im(\pi_Q)$, then  by Proposition \ref{prop-CogrPar} there is an open neighbourhood $\mathcal{U}_S$ of $S$ in $\mathcal{FS}^i(H)^G$ and a symmetry $Q_S$ such that $\mathcal{U}_S\subset\im(\pi_{Q_S})$. Thus, by Lemma \ref{lem-CogrPar}, $\mathcal{U}_S\subset\im(\pi_Q)$ showing that the latter set is open. Secondly, let $\{S_n\}_{n\in\mathbb{N}}\subset\im(\pi_Q)$ be a sequence that converges to some $S\in\mathcal{FS}^i(H)^G$. Again by Proposition \ref{prop-CogrPar}, there is a neighbourhood $\mathcal{U}_S$ of $S$ in $\mathcal{FS}^i(H)^G$ and a symmetry $Q_S$ such that $\mathcal{U}_S\subset\im(\pi_{Q_S})$. Now $S_n\in\mathcal{U}_S$ for sufficiently large $n$, which by Lemma \ref{lem-CogrPar} implies that $\mathcal{U}_S\subset\im(\pi_Q)$ and thus $S\in\im(\pi_Q)$. Consequently, the image of $\pi_Q$ is closed, which finishes the proof.    
\end{proof}
\noindent
In what follows we denote by $B_Q$ the image of the map $\pi_Q$ for a given symmetry $Q$. Moreover, note that if $Q\in\mathcal{FS}^i(H)^G$ is a symmetry and we set $S=Q$ in Proposition \ref{prop-CogrPar}, then \eqref{V} and \eqref{QS} show that $Q_Q=Q$. Consequently, for every symmetry $Q\in\mathcal{FS}^i(H)^G$ there is some $S\in\mathcal{FS}^i(H)^G$ such that $Q_S=Q$. Henceforth we simplify our notation by not specifying $S$ anymore. 

\begin{prop}
The map $\pi_Q:\GL(H)^G\times\mathcal{KS}(H)^G\rightarrow B_Q$ is the projection of a locally trivial fibre-bundle with fibre given by the isotropy group of $Q\in\mathcal{FS}^i(H)^G$.
\end{prop}

\begin{proof}
The proof is almost identical to \cite[Prop. 2.4]{SFLPejsachowiczI}, but we sketch the argument for the convenience of the reader. By Proposition \ref{prop-CogrPar} there is an open subset $\mathcal{U}_Q\subset B_Q$ and a local section $\sigma$ of $\pi_Q$ on $\mathcal{U}_Q$. Then 

\[\eta:\mathcal{U}_Q\times\pi^{-1}_Q(Q)\rightarrow\pi^{-1}_Q(\mathcal{U}_Q),\quad \eta(S,h)=\sigma(S)\cdot h\]
is a local trivialisation over $\mathcal{U}_Q$ with inverse

\[\eta^{-1}(h)=(\pi_Q(h),(\sigma\circ\pi_Q)(h))^{-1}\cdot h).\]
Now this trivialisation can be transported to any point $T\in B_Q$ as follows. As $\pi_Q$ is surjective onto $B_Q$, there is some $\tilde{h}\in \GL(H)^G\times\mathcal{KS}(H)^G$ such that $\pi_Q(\tilde{h})=\tau_{\tilde{h}}(Q)=T$. Then $\mathcal{U}:=\tau_{\tilde{h}}(\mathcal{U}_Q)$ is a neighbourhood of $T$ and $\tau_{\tilde{h}}:\mathcal{U}_Q\rightarrow\mathcal{U}$ is a homeomorphism. Now a local trivialisation over $\mathcal{U}$ is given by

\[\eta':\mathcal{U}\times\pi^{-1}_Q(Q)\rightarrow\pi^{-1}_Q(\mathcal{U}),\quad \eta'(S,h)=\tilde{h}\cdot\sigma(\tau_{\tilde{h}^{-1}}(S))\cdot h,\]
which shows the claim of the proposition.
\end{proof}
\noindent 
Now we finally have everything at hand to prove Theorem \ref{thm-cogredpar} along the lines that we already sketched at the beginning of this section. Let $L=\{L_\lambda\}_{\lambda\in I}$ be a path in $\mathcal{FS}^i(H)^G$. Then the trace $L(I)$ of $L$ is contained in a path component $C$ of $\mathcal{FS}^i(H)^G$. Now let $S\in C$ be arbitrary and let $Q$ be the associated proper symmetry by Propositon \ref{prop-CogrPar}. By the previous proposition, 

\begin{align}\label{piQ_S}
\pi_{Q}:\GL(H)^G\times\mathcal{KS}(H)^G\rightarrow B_{Q}
\end{align}
is the projection of a locally trivial fibre-bundle and clearly $C\subset B_{Q}$ by Corollary \ref{cor-CogrPar}. Let $(E,I,\pi)$ be its pullback by $L$, i.e., the bundle having 

\[E=\{(\lambda,h)\in I\times(\GL(H)^G\times\mathcal{KS}(H)^G):\, L_\lambda=\pi_Q(h)\}\]     
as total space and as bundle projection $\pi$ the restriction of the projection onto the first component. Note that the projection onto the second component $I\times(\GL(H)^G\times\mathcal{KS}(H)^G)\rightarrow\GL(H)^G\times\mathcal{KS}(H)^G$ yields a bundle map from $E$ to the total space of \eqref{piQ_S}. By composing with this map, sections of $(E,I,\pi)$ yield liftings of $L$, and thus the desired map $\tilde{L}:I\rightarrow\GL(H)^G\times\mathcal{KS}(H)^G$ such that $L_\lambda=\pi_Q\circ\tilde{L}_\lambda$ for all $\lambda\in I$. Now $(E,I,\pi)$ is a bundle over the contractible space $I$ and thus trivial. As the triviality of the bundle implies the existence of a globally defined section, this proves Theorem \ref{thm-cogredpar}.\\
As we have pointed out before, the main difficulty in the above argument in comparison to \cite{SFLPejsachowiczI} is that $B_Q$ can be different from $\mathcal{FS}^i(H)^G$. The applications in this paper in Section \ref{sect-simpEx} deal with the rather simple case of a $G=\mathbb{Z}_2$-action. We conclude this section by an example which shows that even in this case $B_Q\neq\mathcal{FS}^i(H)^G$ is possible. Let $H$ be an infinite dimensional Hilbert space and consider on $H\oplus H$ the $\mathbb{Z}_2$-action which maps $(u,v)$ to $(u,-v)$ by its non-trivial element. Then every equivariant operator is of diagonal form. If we now take the proper symmetry $Q(u,v)=(u,-v)$, then we obtain for $M=\diag(A,B)\in\GL(H\oplus H)^{\mathbb{Z}_2}$ and $K=\diag(C,D)\in\mathcal{KS}(H\oplus H)^{\mathbb{Z}_2}$ 

\[\pi_Q(M,K)=\begin{pmatrix}
AA^\ast+C&0\\
0&-BB^\ast+D
\end{pmatrix}\]
and thus every element in $B_Q$ is of the form $\diag(S,T)$, where the essential spectrum of $S$ is on the positive half-line and the essential spectrum of $T$ is on the negative half-line. Thus the proper symmetry $\tilde{Q}:=-Q$, $\tilde{Q}(u,v)=(-u,v)$ is not an element of $B_Q$.


\section{Proof of Theorem \ref{thm:main}}\label{sect-thm:main}
We note at first that it suffices to prove the theorem in the case that the Hessians $L_\lambda$ are strongly indefinite, i.e., $L_\lambda\in\mathcal{FS}^i(H)^G$, $\lambda\in I$. Indeed, if $f:I\times H\rightarrow\mathbb{R}$ is a family of $G$-invariant functionals as in Theorem \ref{thm:main} such that $L_\lambda\in\mathcal{FS}(H)^G$, $\lambda\in I$, then consider the family of functionals $\overline{f}:I\times H\times H\times H\rightarrow\mathbb{R}$ given by

\[\overline{f}_\lambda(w,u,v)=f_\lambda(u)+\frac{1}{2}\|w\|^2-\frac{1}{2}\|v\|^2.\]
Clearly, this family has the same bifurcation points of critical points as $f$, and $\overline{f}$ is $G$-invariant under the orthogonal action $g(w,u,v)=(w,gu,v)$. Finally, the corresponding path of Hessians $\overline{L}=\{\overline{L}_\lambda\}_{\lambda\in I}$ has the same $G$-equivariant spectral flow by \eqref{sflAdditivity}.\\
Thus we henceforth assume that $L_\lambda\in\mathcal{FS}^i(H)^G$, $\lambda\in I$, and obtain from Theorem \ref{thm-cogredpar} a $G$-equivariant cogredient parametrix for $L$, i.e., a path  $M:I\rightarrow\GL(H)^G$ such that

\begin{align}\label{proofCogred}
M^\ast_\lambda L_\lambda M_\lambda= Q+K_\lambda,\qquad \lambda\in I,
\end{align}
where $K_\lambda$ are $G$-equivariant and compact, and $Q\in\mathcal{FS}^i(H)^G$ is a $G$-equivariant symmetry. Note that the functionals of the family $\tilde{f}:I\times H\rightarrow\mathbb{R}$, $\tilde{f}_\lambda(u)=f_\lambda(M_\lambda u)$ are $G$-invariant and $\nabla\tilde{f}_\lambda(u)=M^\ast_\lambda(\nabla f_\lambda)(M_\lambda u)$ so that $\tilde{f}_\lambda$ and $f_\lambda$ have the same bifurcation points of critical points. Moreover, the Hessians $\tilde{L}_\lambda$ of $\tilde{f}_\lambda$ are given by $\tilde{L}_\lambda=M^\ast_\lambda L_\lambda M_\lambda$. Note that $\tilde{L}_\lambda\in\mathcal{FS}(H)^G$, $\lambda\in I$ and thus the $G$-equivariant spectral flow is defined.  

\begin{lemma}
 The paths of operators $L$ and $\tilde{L}$ from above have the same $G$-equivariant spectral flow, i.e.,

\[\sfl_G(L)=\sfl_G(\tilde{L})\in RO(G).\]
\end{lemma}

\begin{proof}
We note at first that $\tilde{L}$ is homotopic to the path $\{M^\ast_0L_\lambda M_0\}_{\lambda\in I}$ and the corresponding homotopy does not affect the spectral flow by \eqref{sflHomotopy} as $L_0, L_1\in\GL(H)$ by the assumptions of Theorem \ref{thm:main}. Now consider the polar decomposition $M_0=UR$ of $M_0$, where $U=M_0(M^\ast_0M_0)^{-\frac{1}{2}}$ is orthogonal and $R=(M^\ast_0M_0)^\frac{1}{2}$ is selfadjoint and positive. Moreover, $U$ and $R$ are $G$-equivariant by Lemma \ref{FunctCalc}. We have

\[M^\ast_0L_\lambda M_0=RU^\ast L_\lambda UR,\quad\lambda\in I,\]
and see that the homotopy 

\[\{((1-s)R+sI_H)U^\ast L_\lambda U((1-s)R+sI_H)\}_{(s,\lambda)\in I\times I}\]
deforms $\{M^\ast_0L_\lambda M_0\}_{\lambda\in I}$ into the path $\{U^\ast L_\lambda U\}_{\lambda\in I}$. Note that also this homotopy does not affect the spectral flow by \eqref{sflHomotopy} as $L_0, L_1\in\GL(H)$ and $(1-s)R+sI_H\in GL(H)$ for all $s\in I$. Finally, for any $a>0$,

\[U^\ast:E(L_\lambda,[0,a])\rightarrow E(U^\ast L_\lambda U,[0,a]),\quad\lambda\in I,\]
is a $G$-equivariant isomorphism and thus

\[[E(L_\lambda,[0,a])]=[E(U^\ast L_\lambda U,[0,a])],\quad\lambda\in I.\]
Consequently, it follows from the definition \eqref{sfl-equiv} that $\{U^\ast L_\lambda U\}_{\lambda\in I}$ and $L$ have the same $G$-equivariant spectral flow, which proves the lemma.  
\end{proof}
\noindent
In summary, we can henceforth assume without loss of generality that $L_\lambda=Q+K_\lambda$, $\lambda\in I$, where the operators $K_\lambda$ are compact, selfadjoint and $G$-equivariant, and $Q=P-(I_H-P)$ for some $G$-equivariant orthogonal projection $P$ having infinite dimensional kernel and range.

\subsection{Reduction to Finite Dimensions I}

We begin by a technical lemma on decomposing $H$ into finite dimensional invariant subspaces. Note that in case of a trivial group action, this is merely the existence of an orthonormal basis.

\begin{lemma}
There is a sequence of finite-dimensional $G$-invariant subspaces $H_n\subset H$, $n\in\mathbb{N}$, such that 

\[H_n\subset H_{n+1}\quad\text{and}\quad P_nu\xrightarrow{n\rightarrow\infty} u\,\,\text{for all}\,\, u\in H,\]
where $P_n$ denotes the orthogonal projection onto $H_n$.
\end{lemma}

\begin{proof}
Let $\mathcal{F}$ be the set of all subsets $B\subset H$ such that

\begin{itemize}
 \item[(i)] $\|x\|=1$ for all $x\in B$,
 \item[(ii)] $\langle x,y\rangle=0$ for all $x,y\in B$, $x\neq y$,
 \item[(iii)] for all $x\in B$ there exists a subspace $V\subset H$ of finite dimension such that 
 \[Gx:=\{gx:\, g\in G\}\subset V.\]
\end{itemize}
Note that $\mathcal{F}$ is not empty as every representation of $G$ on an infinite dimensional Banach space has a finite dimensional subrepresentation by \cite[Cor. 5.4 (a)]{IzeVignoli}. We now partially order $\mathcal{F}$ by inclusion and aim to use Kuratowski-Zorn lemma. If $\mathcal{E}\subset\mathcal{F}$ is totally ordered, then the union of all elements in $\mathcal{E}$ satisfies $(i)-(iii)$ from above and thus is an upper bound for $\mathcal{E}$. Consequently, there exists a maximal element $B^\ast$ of $\mathcal{F}$. As $B^\ast$ is orthonormal and $H$ is separable, $B^\ast$ is countable, say $B^\ast=\{e_1,e_2,\ldots\}$. We now let $H_n$ be the intersection of all $G$-invariant subspaces of $H$ that contain $\{e_1,\ldots,e_n\}$. Note that $H_n$ is of finite dimension by $(iii)$. Moreover,

\[U:=\overline{\bigcup^\infty_{n=1}H_n}\]
is a $G$-invariant subspace of $H$. Now assume that $U\neq H$. Then $U^\perp$ is a $G$-invariant subspace and thus contains a finite dimensional subrepresentation. The latter claim is trivial if $U^\perp$ is of finite dimension, and otherwise again follows by \cite[Cor. 5.4 (a)]{IzeVignoli}. Now we take an element $v$, $\|v\|=1$, of this finite dimensional subrepresentation of $U^\perp$. Then $B^\ast\cup\{v\}\in\mathcal{F}$ is larger than $B^\ast$ which contradicts the maximality. Thus $U=H$, which in particular implies that $(P_n)_{n\in\mathbb{N}}$ weakly converges to the identity.        
\end{proof}
\noindent
Let us recall that $Q=P-(I_H-P)$ for some orthogonal projection $P$ having infinite dimensional kernel and range.

\begin{cor}
There is a sequence of finite dimensional $G$-invariant subspaces $H_n\subset H$, $n\in\mathbb{N}$, such that 

\[H_n\subset H_{n+1},\quad  [P_n,Q]=0,\,n\in\mathbb{N}, \quad\text{and}\quad P_nu\xrightarrow{n\rightarrow\infty} u\,\,\text{for all}\,\, u\in H,\]
where $P_n$ denotes the orthogonal projection onto $H_n$.
\end{cor}

\begin{proof}
We denote by $H^+$ the image of $P$ and by $H^-$ its kernel, which are both invariant under $G$ and of infinite dimension. Let $H^\pm_n$ be the corresponding spaces and $P^\pm_n$ the projections as given by the previous lemma. We set $H_n:=H^+_n\oplus H^-_n$ and note that $P_n:=P^+_n+P^-_n$ is the orthogonal projection onto $H_n$ if we regard $P^\pm_n$ as orthogonal projection on $H$ with kernel extended to $H^\mp$. Now the first and the third claimed property are satisfied. The remaining one follows from

\begin{align*}
P_nQ-QP_n&=(P^+_n+P^-_n)(P-(I_H-P))-(P-(I_H-P))(P^+_n+P^-_n)\\
&=P^+_nP-P^-_n(I_H-P)-PP^+_n+(I_H-P)P^-_n=0,
\end{align*}
where we use that $P_1P_2=P_2P_1=P_1$ for orthogonal projections $P_1$, $P_2$ such that $\im(P_1)\subset\im(P_2)$.
\end{proof}
\noindent
Note that as $P_n$ commutes with $Q$ by the previous lemma, it follows that $Q(H_n)=H_n$ as well as $Q(H^\perp_n)=H^\perp_n$.

\begin{lemma}\label{bif-lemma-isomorphisms}
There is $n_0\in\mathbb{N}$ such that for all $n\geq n_0$

\begin{itemize}
\item[(i)] $(I_H-P_n)L_\lambda\mid_{H^\perp_n}\in\GL(H^\perp_n),\quad \lambda\in I$,
\item[(ii)] $sL_\lambda+(1-s)((I_H-P_n)L_\lambda(I_H-P_n)+P_nL_\lambda P_n)\in\GL(H),\quad \lambda=0,1,\, s\in I.$
\end{itemize}
\end{lemma}

\begin{proof}
We note at first that $$(I_H-P_n)L_\lambda\mid_{H^\perp_n}=Q+(I_H-P_n)K_\lambda\mid_{H^\perp_n}$$ is a compact perturbation of an invertible operator and thus Fredholm of index $0$. Consequently, to prove the first assertion, we only need to show that $(I_H-P_n)L_\lambda\mid_{H^\perp_n}$ is injective.\\
Since $\{K_\lambda\}_{\lambda\in I}$ is a continuous family of compact operators, the set $\{K_\lambda(u):\,\lambda\in I,\, \|u\|=1\}$ is relatively compact. Therefore, as $I_H-P_n$ uniformly converges to $0$ on compact subsets of $H$, there exists $n_0\in\mathbb{N}$ such that

\[\|(I_H-P_n)K_\lambda u\|\leq\frac{1}{2}\|u\|,\quad u\in H,\, \lambda\in I,\, n\geq n_0.\]
Moreover, $\|Qu\|=\|u\|$, $u\in H$, as $Q$ is orthogonal, and thus

\begin{align*}
\|(I_H-P_n)L_\lambda u\|=\|Qu+(I_H-P_n)K_\lambda u\|\geq\frac{1}{2}\|u\|,\quad u\in H^\perp_n,
\end{align*}
showing the injectivity of $(I_H-P_n)L_\lambda\mid_{H^\perp_n}$.\\
To show (ii), we note at first that by a simple calculation

\begin{align*}
&sL_\lambda+(1-s)((I_H-P_n)L_\lambda(I_H-P_n)+P_nL_\lambda P_n)\\
&=Q+sK_\lambda+(1-s)((I_H-P_n)K_\lambda(I_H-P_n)+P_nK_\lambda P_n),
\end{align*}
which are all Fredholm operators of index $0$. We now assume by contradiction that $n_0$ as in the assertion does not exist. Then there are sequences $(u_n)_{n\in\mathbb{N}}$, $\|u_n\|=1$, and $(s_n)_{n\in\mathbb{N}}$ such that

\[Qu_n+s_nK_0u_n+(1-s_n)((I_H-P_n)K_0(I_H-P_n)u_n+P_nK_0 P_nu_n)=0,\quad n\in\mathbb{N}.\]
As $K_0$ is compact and $P_n$ converges on compact subsets of $H$ to the identity, we see that there is a convergent subsequence of $(Qu_n)$. Henceforth, we denote this sequence by the same indices and assume as well that $(s_n)$ converges to some $s^\ast\in I$. It follows from the invertibility of $Q$ that $(u_n)$ converges to some $u\in H$ of norm $1$. Thus
\[\lim_{n\rightarrow\infty} (I_H-P_n)K_0(I_H-P_n)u_n=0,\quad \lim_{n\rightarrow\infty}P_nK_0 P_n u_n=K_0 u,\]
and so
\[L_0 u=Qu+K_0u=Qu+s^\ast K_0 u+(1-s^\ast)K_0 u=0\]
in contradiction to the invertibility of $L_0$. Of course, the same argument applies to the invertible operator $L_1$.
\end{proof}
\noindent
We now set $L^n_\lambda:= P_nL_\lambda\mid_{H_n}:H_n\rightarrow H_n$ and note that these operators are $G$-equivariant. It follows from \eqref{sflNormalisation}, \eqref{sflAdditivity} and Proposition \ref{prop-sfldiffMorse} that for $n\geq n_0$

\begin{align}\label{eqref-finreduction}
\sfl_G(L)=\sfl_G(L^n)=[E^-(L_1)]-[E^-(L_0)]\in RO(G)
\end{align}
and thus the Hessians are reduced to finite dimensions.


\subsection{Reduction to Finite Dimensions II}

For reducing the nonlinear problem to finite dimensions, we need the following technical lemma from \cite[Lem. 3.7]{MJN21} that was shown in the non-equivariant case in \cite{SFLPejsachowiczI} and \cite{BifJac}.

\begin{lemma}\label{lemma-implicit}
Let $H$ be a real Hilbert space and $G$ a compact Lie group acting orthogonally on $H$. Let $U\subset H$ be an open invariant subset of $H$ containing $0\in U$ and $f:I\times U\rightarrow\mathbb{R}$ a continuous one-parameter family of $G$-invariant $C^2$-functionals. Let $F(\lambda,u):=(\nabla f_\lambda)(u)$ and assume that $F(\lambda,0)=0$ for all $\lambda\in I$. Suppose that there is an orthogonal decomposition $H=X\oplus Y$, where $X$ is $G$-invariant and of finite dimension, and such that for 

\[F(\lambda,u)=(F_1(\lambda,x,y),F_2(\lambda,x,y))\in X\oplus Y,\quad u=(x,y)\in X\oplus Y,\]
we have that $(D_y F_2)(\lambda,0,0):Y\rightarrow Y$ is invertible for all $\lambda\in I$. Then:

\begin{enumerate}
 \item[(i)] There are an open ball $B_X=B(0,\delta)\subset X$ and a unique continuous family of equivariant $C^1$-maps $\eta:I\times B_X\rightarrow Y$ such that $\eta(\lambda,0)=0$ for all $\lambda\in I$, and 
 
 \begin{align}\label{implicit}
 F_2(\lambda,x,\eta(\lambda,x))=0,\quad (\lambda,x)\in I\times B_X.
 \end{align}
 
 \item[(ii)] Let the family of functionals $\overline{f}:I\times B_X\rightarrow\mathbb{R}$ and the map $\overline{F}:I\times B_X\rightarrow X$ be defined by
 
 \[\overline{f}(\lambda,x)=f(\lambda,x,\eta(\lambda,x)),\qquad \overline{F}(\lambda,x)=F_1(\lambda,x,\eta(\lambda,x)).\]
 Then $\overline{f}$ is a continuous family of $G$-invariant $C^2$-functionals on $B_X$ and $$\nabla\overline{f}(\lambda,x)=\overline{F}(\lambda,x),\qquad (\lambda,x)\in I\times B_X,$$
 which is a $G$-equivariant map.
\end{enumerate}
\end{lemma}
\noindent
We now set $X=H_n$, $Y=H^\perp_n$ and consider the splitting $F=(F^n_1,F^n_2)$, where

\[F^n_1(\lambda,u,v)=P_n F(\lambda,u,v),\qquad F^n_2(\lambda,u,v)=(I_H-P_n) F(\lambda,u,v).\]
As $D_vF^n_2(\lambda,0,0)=(I_H-P_n)L_\lambda\mid_{H^\perp_n}:H^\perp_n\rightarrow H^\perp_n$ is an isomorphism for $n\geq n_0$ by Lemma \ref{bif-lemma-isomorphisms}, we obtain from the previous lemma a family of $G$-invariant functionals $\overline{f}:I\times B_{H_n}\rightarrow\mathbb{R}$ for some ball $B_{H_n}\subset H_n$ such that every bifurcation point of critical points of $\overline{f}$ is a bifurcation point of $f$. Consequently, our aim is now to show that $\overline{f}$ has a bifurcation of critical points from the trivial branch if \eqref{eqref-finreduction} is non-trivial in $RO(G)$. 

\begin{prop}
For the Hessians $\overline{L}^n_\lambda$ of the $G$-invariant functionals $\overline{f}_\lambda$ at $0\in H_n$, there exists $n_1\geq n_0$ such that $\overline{L}^n_\lambda$ is invertible and

\[[E^-(\overline{L}^n_\lambda)]=[E^-(L^n_\lambda)],\quad \lambda=0,1,
\]
for $n\geq n_1$.
\end{prop}

\begin{proof}
Let $\eta^n_\lambda:B_{H_n}\rightarrow H^\perp_n$ be the continuous family of $C^1$-maps from Lemma \ref{lemma-implicit} for the splitting $H=H_n\oplus H^\perp_n$, and set $A^n_\lambda:=D_0 \eta^n_\lambda$. Note that $A^n_\lambda$ is $G$-equivariant. By differentiating \eqref{implicit} implicitly, we obtain

\[A^n_\lambda=-(D_v F^n_{2}(\lambda,0,0))^{-1} D_u F^n_{2}(\lambda,0,0)=-((I_H-P_n)L_\lambda\mid_{H^\perp_n})^{-1}(I_H-P_n)L_{\lambda}\mid_{H_n}.\]
In the first part of the proof of Lemma \ref{bif-lemma-isomorphisms} we obtained

\[\|(I_H-P_n)L_\lambda u\|\geq \frac{1}{2}\|u\|,\quad u\in H^\perp_n,\, n\geq n_0,\]
which shows that

\[\|((I_H-P_n)L_\lambda\mid_{H^\perp_n})^{-1}\|\leq 2,\quad n\geq n_0,\, \lambda=0,1.\]
Using once again that $L_\lambda=Q+K_\lambda$ and 

\begin{align}\label{cpctto0}
\|(I_H-P_n)K_\lambda\|\rightarrow 0,\quad n\rightarrow\infty,
\end{align}
by the compactness of $K_\lambda$, this yields

\begin{align}\label{Ato0}
\|A^n_\lambda\|\leq 2 \|(I_H-P_n)(Q+K_\lambda)\mid_{H_n}\|\leq 2 \|(I_H-P_n)K_\lambda\|\rightarrow 0,\quad n\rightarrow\infty,
\end{align}
which we note for later reference.\\
We now consider $\overline{L}^n_\lambda$ and note at first that

\[\overline{L}^n_\lambda=P_nL_\lambda(I_{H_n}+A^n_\lambda)=L^n_{\lambda}+P_nL_\lambda A^n_\lambda.\]
Let us introduce two paths $\{M^n_{(s,\lambda)}\}_{s\in I}$, $\lambda=0,1$, of $G$-equivariant selfadjoint operators on $H_n$ by

\[M^n_{(s,\lambda)}=L^n_{\lambda}+sP_nL_\lambda A^n_\lambda.\]
We now aim to find $n_1\in\mathbb{N}$ such that $M^n_{(s,0)}$ and $M^n_{(s,1)}$ are invertible for all $s\in I$ and $n\geq n_1$.\\
We first note that there is $k_1\in\mathbb{N}$ and a constant $C>0$ such that for $\lambda=0,1$ and all $n\geq k_1$

\begin{align}\label{eqref-invest}
\|L^n_\lambda u\|=\|P_nL_\lambda u\|\geq C\|u\|,\quad u\in H_n.
\end{align}
Indeed, as $L_\lambda$ is invertible for $\lambda=0,1$, there is a constant $C>0$ such that

\[\|L_\lambda u\|\geq 2C\|u\|,\quad u\in H,\, \lambda=0,1.\]
Now by direct computation

\[P_nL_\lambda u=L_\lambda u-(I_H-P_n)K_\lambda u,\quad u\in H_n,\]
and \eqref{cpctto0} implies that there is $k_1\in\mathbb{N}$ such that for $n\geq k_1$

\[\|(I_H-P_n)K_\lambda u\|\leq C\|u\|,\quad u\in H,\]
which shows \eqref{eqref-invest}. Finally, by \eqref{Ato0} there is $k_2\in\mathbb{N}$ such that $\|L_\lambda\|\|A^n_\lambda\|\leq C$ for all $n\geq k_2$, $\lambda=0,1$. Consequently, if $n\geq n_1:=\max\{k_1,k_2\}$,

\begin{align*}
\|M^n_{(s,\lambda)}u\|\geq \|L^n_\lambda u\|-s\|P_nL_\lambda A^n_\lambda u\|\geq 2C\|u\|-\|L_\lambda\|\|A^n_\lambda\| \|u\|\geq C\|u\|
\end{align*}
for $\lambda=0,1$ and $0\leq s\leq 1$. Thus $M^n_{(s,\lambda)}:H_n\rightarrow H_n$ is injective and hence invertible on the finite dimensional space $H_n$. Note that the proposition is shown if we prove that

\begin{align}\label{eqreffinal}
[E^-(M^n_{(0,\lambda)})]=[E^-(M^n_{(1,\lambda)})],\quad n\geq n_1,
\end{align}
for $\lambda=0,1$. As $M^n_{(s,\lambda)}$ is invertible for all $s\in I$, the maps $[0,1]\ni s\mapsto\chi_{(-\infty,0)}(M^n_{(s,\lambda)})\in\mathcal{L}(H_n)$ are continuous. Thus there is a partition $0=s_0\leq s_1\leq\ldots\leq s_k=1$ such that 

\begin{align}\label{eqreffinalII}
\|\chi_{(-\infty,0)}(M^n_{(s_j,\lambda)})-\chi_{(-\infty,0)}(M^n_{(s_{j-1},\lambda)})\|<1.
\end{align}
Moreover, these projections are $G$-equivariant as their images are invariant and $G$ acts orthogonally. We now shorten our notation by setting $P:=\chi_{(-\infty,0)}(M^n_{(s_j,\lambda)})$, $Q:=\chi_{(-\infty,0)}(M^n_{(s_{j-1},\lambda)})$, and we claim that $\im(P)$ and $\im(Q)$ are isomorphic as $G$-representations. To prove this, we first note that the $G$-equivariant map $U:=PQ+(I_H-P)(I_H-Q)$ maps $\im(P)$ into $\im(Q)$. Moreover, a direct computation yields

\[(QP+(I_H-Q)(I_H-P))U=I_H-(P-Q)^2.\]
As $\|P-Q\|<1$ by \eqref{eqreffinalII}, the right hand side is an isomorphism and consequently $U$ is injective. Thus $U\mid_{\im(P)}:\im(P)\rightarrow\im(Q)$ is a $G$-equivariant isomorphism and so $[E^-(M^n_{(s_j,\lambda)})]=[E^-(M^n_{(s_{j-1},\lambda)})]$ for $j=1,\ldots,k$. Thus \eqref{eqreffinal} is shown, which eventually finishes the proof of the proposition. 
\end{proof}
\noindent
In conclusion, by \eqref{eqref-finreduction} and the previous proposition, we have reduced Theorem \ref{thm:main} to finite dimensions, i.e., we only need to prove it under the additional assumption that $\dim(H)<\infty$.


\subsection{Equivariant Conley Index and End of the Proof}
The aim of this final step of the proof is to show Theorem \ref{thm:main} under the additional assumption that $\dim(H)<\infty$. The proof is based on the equivariant Conley index, for which we mainly follow Bartsch's monograph \cite{BartschBook}.\\
Let $\phi_\lambda:\mathbb{R}\times H\rightarrow H$ be the flow of the equation

\begin{align}\label{conleyequ}
u'(t)=-(\nabla f_\lambda)(u(t))
\end{align} 
and note that its stationary solutions are the critical points of $f_\lambda$. Here we assume without loss of generality that the flow is global, which can be achieved by multiplying $f_\lambda$ by a smooth cut-off function in a neighbourhood of $0\in H$ and this does not affect the existence of bifurcation of critical points from $0\in H$. Note that $\varphi_\lambda(t,\cdot):H\rightarrow H$ is equivariant.\\
For a $G$-invariant subset $U\subset H$ we denote by

\[\inv(U,\varphi_\lambda)=\{u\in H:\, \varphi_\lambda(t,u)\in U\,\text{ for all } t\in\mathbb{R}\}\] 
the maximal (flow-)invariant subset of $U$, which clearly is $G$-invariant as well. A compact invariant set $S\subset H$ is called isolated if there is a compact $G$-invariant neighbourhood $U$ of $S$ such that $S=\inv(U,\varphi_\lambda)$ and $S \subset \interior U$. In this case $U$ is called an isolating neighbourhood of $S$. If $S\subset H$ is an isolated invariant set, then a pair $(N_1,N_0)$ of compact $G$-invariant subsets $N_0\subset N_1$ is called a $G$-index pair for $S$ if

\begin{itemize}
 \item $\overline{N_1\setminus N_0}$ is an isolating neighbourhood of $S$,
 \item $N_0$ is positively invariant with respect to $N_1$, which means that if $u\in N_0$ and $\varphi_\lambda(t,u)\in N_1$ for all $0\leq t\leq t'$, then $\varphi_\lambda(t,u)\in N_0$ for all $0\leq t\leq t'$,
 \item $N_0$ is an exit set for $N_1$, which means that if $u\in N_1$ and $\varphi_\lambda (t,u)\notin N_1$ for some $t>0$ then there is $t'\in[0,t)$ such that $\varphi_\lambda(t',u)\in N_0$ and $\varphi_\lambda([0,t'],u) \subset N_1$.
\end{itemize}  
If $U\subset H$ is an isolating neighbourhood for the flow $\varphi_\lambda$, then there is a $G$-index pair for $S = \inv U$, and if $(N_1,N_0)$, $(N'_1,N'_0)$ are two $G$-index pairs for $S$, then the quotient spaces $N_1/N_0$ and $N'_1/N'_0$ are homotopy equivalent by a base point preserving $G$-equivariant homotopy equivalence. Thus it is sensible to define the $G$-equivariant Conley index $\mathcal{C}(U,\varphi_\lambda)$ of $S$ as the based $G$-homotopy type $[N_1/N_0, [N_0]]$, where $(N_1,N_0)$ is any $G$-index pair for $S$. Finally, let us recall that by the continuation theorem for the Conley index $\mathcal{C}(U,\varphi_0)=\mathcal{C}(U,\varphi_1)$ if $U$ is an isolating neighbourhood for $\varphi_\lambda$ for all $\lambda\in I$.\\
Let us now come back to bifurcation of critical points and let us recall that $u\equiv 0\in H$ is a stationary solution of \eqref{conleyequ} for all $\lambda\in I$. Suppose that there is no bifurcation point. Since any isolated critical point is an isolated invariant set, there exists $\epsilon > 0$ such that $U = D(0,\epsilon)$ is an isolating neighbourhood for all $\varphi_\lambda$, $\lambda \in I$. This implies that $\mathcal{C}(U,\varphi_0) = \mathcal{C}(U,\varphi_1)$. On the other hand, we know that $[E^-(L_0)]\neq [E^-(L_1)]$, i.e., these spaces are non-isomorphic $G$-representations. As $G$ is nice, this implies that the quotients $D_0/\partial D_0$ and $D_1/\partial D_1$ are not $G$-homotopic, where $D_\lambda$ denotes the unit disc of $E^-(L_\lambda)$ for $\lambda=0,1$. In our case $\mathcal{C}(U,\varphi_\lambda) = [D_{\lambda}/\partial D_\lambda, [\partial D_\lambda]]$ for $\lambda = 0,1$. This contradicts the equality $\mathcal{C}(U,\varphi_0) = \mathcal{C}(U,\varphi_1)$, and consequently the assumption that there are no bifurcation points.

\section{Examples of Bifurcation of Critical Points of Even Functionals}\label{sect-simpEx}
The aim of this section is illustrate Theorem \ref{thm:main} by various examples of functionals $f_\lambda$ that are invariant under an action of $\mathbb{Z}_2$. As all real irreducible representations of $\mathbb{Z}_2$ are one dimensional, every real $k$-dimensional representation is up to isomorphism a $k\times k$ diagonal matrix of the form $\diag(1,\ldots,1,-1,\ldots,-1)$. Thus we obtain an isomorphism $\phi:RO(\mathbb{Z}_2)\rightarrow\mathbb{Z}\oplus\mathbb{Z}$ of abelian groups by setting
 
 \begin{align}\label{phi}
 \phi([E]-[F])=(\dim(E)-\dim(F),\dim(E^G)-\dim(F^G)),
 \end{align}
 where $E^G\subset E$ and $F^G\subset F$ denote the spaces of fixed points under the group action.

\begin{lemma}
Let $H$ be a real separable Hilbert space on which $G=\mathbb{Z}_2$ acts orthogonally. Then for every path $L=\{L_\lambda\}_{\lambda\in I}$ in $\mathcal{FS}(H)^G$

\begin{align}\label{phi2}
\phi(\sfl_G(L))=(\sfl(L),\sfl(L\mid_{H^G}))\in\mathbb{Z}\oplus\mathbb{Z},
\end{align}
where $H^G$ is the fixed point set of the action of $G$. Moreover,

\begin{align}\label{phi3}
\sfl(L)=\sfl(L\mid_{H^G})+\sfl(L\mid_{(H^G)^\perp}).
\end{align}
\end{lemma}

\begin{proof}
Note that $H^G$ reduces the operators $L_\lambda$ and thus we indeed obtain a path of selfadjoint operators $L\mid_{H^G}=\{L_\lambda\mid_{H^G}\}_{\lambda\in I}$ that all have finite dimensional kernels. Moreover, $\im(L_\lambda\mid_{H^G})=\im(L_\lambda\mid_{H^G\cap(\ker L_\lambda)^\perp})$ and the latter set is closed in $\im(L_\lambda)$ as $L_\lambda\mid_{(\ker L_\lambda)^\perp}:(\ker L_\lambda)^\perp\rightarrow\im(L_\lambda)$ is a homeomorphism. Consequently, $\im(L_\lambda\mid_{H^G})$ is closed in $H$ and thus in $H^G$. Therefore the operators $L_\lambda\mid_{H^G}$ are in $\mathcal{FS}(H^G)$ and so $\sfl(L\mid_{H^G})$ is defined. Likewise the restriction $L\mid_{(H^G)^\perp}$ to the invariant subspace $(H^G)^\perp$ is an element of $\mathcal{FS}((H^G)^\perp)$, and now \eqref{phi3} follows from \eqref{sflAdditivity}.\\
Finally, \eqref{phi2} is a simple consequence of \eqref{sfl}, \eqref{sfl-equiv} and \eqref{phi} when noting that 

\[E(L_\lambda,[0,a])^G=H^G\cap E(L_\lambda,[0,a])=E(L_\lambda\mid_{H^G},[0,a])\]
for any $a>0$.       
\end{proof}
\noindent
In the examples below we need a common method to compute the classical spectral flow \eqref{sfl} that we now want to recap (see \cite{Robbin-Salamon}, \cite{Homoclinics}). Let $L=\{L_\lambda\}_{\lambda\in I}$ be a path in $\mathcal{FS}(H)$ that is continuously differentiable in the parameter $\lambda$. We call $\lambda\in I$ a crossing if $\ker(L_\lambda)\neq\{0\}$, and the associated crossing form is the quadratic form defined by

\begin{align*}
\Gamma(L,\lambda)[u]=\langle \dot{L}_\lambda u,u\rangle,\quad u\in\ker(L_\lambda),
\end{align*}
where $\dot{L}_\lambda$ denotes the derivative with respect to $\lambda$. A crossing $\lambda\in I$ is regular if $\Gamma(L,\lambda)$ is non-degenerate. Regular crossings are isolated and thus every path $L$ parametrised by a compact interval $I$ can only have finitely many of them. Finally, if $L=\{L_\lambda\}_{\lambda\in I}$ has only regular crossings, then the spectral flow \eqref{sfl} is given by

\begin{align}\label{crossing-form}
\sfl(L)=\sum_{\lambda\in I}\sgn(\Gamma(L,\lambda)),
\end{align}
where $\sgn(\Gamma(L,\lambda))$ denotes the signature of the quadratic form $\Gamma(L,\lambda)$.


\subsection{A system of elliptic PDEs}
The aim of this section is to consider simple settings in which the classical spectral flow fails to show the existence of a bifurcation, whereas the $G$-equivariant spectral flow succeeds. Let us consider on a bounded smooth domain $\Omega\subset\mathbb{R}^N$ indefinite elliptic systems of the type

\begin{align}
			\label{eq:pdeGeneral2}
			\left\{
			\begin{array}{rl}
				-\Delta u(x)&=\nabla_{u} F(\lambda,x,u(x),v(x))\ \hspace*{0.25cm} \mathrm{in}\ \Omega\\		
				\Delta v(x)&=\nabla_{v} F(\lambda,x,u(x),v(x))\ \hspace*{0.25cm} \mathrm{in}\ \Omega\\	
				&u(x)=v(x)=0\ \hspace*{1.3cm}  \mathrm{on}\ \partial \Omega,
			\end{array}
			\right.
		\end{align}
where $F:I\times\overline{\Omega}\times\mathbb{R}^2\rightarrow\mathbb{R}$ is a continuous function such that the partial derivatives $\nabla_{u} F$, $\nabla_{v} F$ exist and depend continuously on $(\lambda,x,u,v)\in I\times\overline{\Omega}\times\mathbb{R}^2$. Moreover, we assume that $\nabla F(\lambda,x,0)=0$ so that the constant function $(u,v)=(0,0)$ is a solution of \eqref{eq:pdeGeneral2} for all $\lambda\in I$. It is well known that, under a common growth condition on the nonlinearity $F$, the solutions of \eqref{eq:pdeGeneral2} are the critical points of the functional $f_\lambda:H^1_0(\Omega,\mathbb{R}^2)\rightarrow\mathbb{R}$ defined by

\begin{align}\label{examplefcritpoint}
f_\lambda(w)&=\frac{1}{2}\int_\Omega{\langle\nabla u(x),\nabla u(x)\rangle\,dx}-\frac{1}{2}\int_\Omega{\langle \nabla v(x),\nabla v(x)\rangle\,dx}+\int_\Omega{F(\lambda,x,u(x),v(x))\, dx},
\end{align} 
where $w=(u,v)\in H^1_0(\Omega,\mathbb{R}^2)$. The first derivative of $f_\lambda$ at $w=(u,v)\in H^1_0(\Omega,\mathbb{R}^2)$ is

\begin{align*}
D_wf_\lambda(\overline{w})&=\int_\Omega{\langle\nabla u(x),\nabla \overline{u}(x)\rangle\,dx}-\int_\Omega{\langle \nabla v(x),\nabla\overline{v}(x)\rangle\,dx}\\
&+\int_\Omega{\frac{\partial F}{\partial u}(\lambda,x,u(x),v(x))\,\overline{u}(x)\, dx}+\int_\Omega{\frac{\partial F}{\partial v}(\lambda,x,u(x),v(x))\,\overline{v}(x)\, dx},\quad \overline{w}=(\overline{u},\overline{v}),
\end{align*}
and thus the critical points of \eqref{examplefcritpoint} are the weak solutions of \eqref{eq:pdeGeneral2}. Finally, the second derivatives at the critical point $w=0\in H^1_0(\Omega,\mathbb{R}^2)$ are given by

\begin{align}\label{ex:Hessian}
\begin{split}
D^2_0f_\lambda(w_1,w_2)&=\int_\Omega{\langle\nabla u_1(x),\nabla u_2(x)\rangle\,dx}-\int_\Omega{\langle \nabla v_1(x),\nabla v_2(x)\rangle\,dx}\\
&+\int_\Omega{\langle S_\lambda(x)w_1,w_2\rangle \, dx},\quad w_1=(u_1,v_1), w_2=(u_2,v_2),
\end{split}
\end{align}
where $S_\lambda(x):= D^2_0F(\lambda,x,\cdot,\cdot)$ denotes the Hessian of $F_{(\lambda,x)}:\mathbb{R}^2\rightarrow\mathbb{R}$ at the critical point $0$. It is well known that the Riesz representations $L_\lambda$ of $D^2_0f_\lambda$ are Fredholm and their kernels are the solutions of the linear Dirichlet problems

\begin{align}
			\label{eq:pdeGeneral2LIN}
			\left\{
			\begin{array}{rl}
				\begin{pmatrix} -\Delta&0\\0&\Delta\end{pmatrix}\begin{pmatrix}u\\v\end{pmatrix} &= S_\lambda(x) \begin{pmatrix}u\\v\end{pmatrix}\ \hspace*{0.25cm} \mathrm{in}\ \Omega\\							
				u(x)=v(x) = 0&\ \hspace*{2.3cm}  \mathrm{on}\ \partial \Omega,
			\end{array}
			\right.
		\end{align}
\noindent
(see, e.g., \cite{AleIchIndef}). Thus if \eqref{eq:pdeGeneral2LIN} has only the trivial solution for $\lambda=0,1$ and if $\sfl(L)\neq 0$, then there is a bifurcation of \eqref{eq:pdeGeneral2}. Let us point out that the computation of the spectral flow might in general be difficult.\\
Our first example shall illustrate Theorem \ref{bif-thm-AFi} on the non-existence of bifurcation in case of a vanishing classical spectral flow. We consider \eqref{eq:pdeGeneral2} for $F(\lambda,x,u,v)=\frac{\lambda}{2}(u^2-v^2)+u^3v+v^3u$, i.e.,

\begin{align}
			\label{eq:pdeExAFi}
			\left\{
			\begin{array}{rl}
				-\Delta u(x) &= \lambda u(x)+3u(x)^2v(x)+v(x)^3\ \hspace*{0.25cm} \mathrm{in}\ \Omega\\		
				\Delta v(x) &= -\lambda v(x)+3v(x)^2u(x)+u(x)^3\ \hspace*{0.25cm} \mathrm{in}\ \Omega\\	
				u(x) &=v(x)= 0\ \hspace*{2.5cm}  \mathrm{on}\ \partial \Omega.
			\end{array}
			\right.
		\end{align}
Now $S_\lambda(x)=\diag(\lambda,-\lambda)$ is a diagonal matrix and it follows from \eqref{ex:Hessian} that the crossing forms are given by

\[\Gamma(L,\lambda)=\int_\Omega{u^2-v^2\,dx},\quad (u,v)\in\ker(L_\lambda).\]
The kernel of $L_\lambda$ is made by the solutions of the corresponding equation \eqref{eq:pdeGeneral2LIN}. Now in case of the existence of a solution for $S_\lambda(x)=\diag(\lambda,-\lambda)$ and some $\lambda\in I$, the solution space is two-dimensional and thus $\Gamma(L,\lambda)$ is non-degenerate, but its signature vanishes. Hence $\sfl(L)=0$ by \eqref{crossing-form} and Theorem \ref{thm-FPR} does not apply. If we now multiply the first equation of \eqref{eq:pdeExAFi} by $v$, the second by $u$, add those equations and integrate, we obtain from Green's identity

\[\int_\Omega{6u^2(x)v^2(x)+u^4(x)+v^4(x) dx}=0.\]
Thus $u=v=0$ and there is no bifurcation for \eqref{eq:pdeExAFi}.\\ 
We now consider \eqref{eq:pdeGeneral2} in the case that $F$ is even in $v$, i.e., 

\begin{align}\label{ex-evencond}
F(\lambda,x,u,-v)=F(\lambda,x,u,v)\quad \lambda\in I, x\in\Omega, (u,v)\in\mathbb{R}^2.
\end{align}
Then $S_\lambda$ is a diagonal matrix, $S_\lambda(x)=\diag(a_\lambda(x),b_\lambda(x))$, where $a_\lambda(x)$ and $b_\lambda(x)$ are the second partial derivatives of $F(\lambda,x,\cdot,\cdot)$ with respect to $u$ and $v$ at $(0,0)$. Note that multiplicity of solutions of \eqref{eq:pdeGeneral2} under the condition \eqref{ex-evencond} were studied by Bartsch and Clapp in \cite{Clapp}.  

\begin{prop}\label{prop-ExI}
If \eqref{ex-evencond} holds, \eqref{eq:pdeGeneral2LIN} has only the trivial solution for $\lambda=0,1$ and the boundary value problems

\begin{align}\label{eq-prop-ExI}
			\left\{
			\begin{array}{rl}
				-\Delta u(x) &= a_\lambda(x)u(x)\ \hspace*{0.25cm} \mathrm{in}\ \Omega\\			
				u(x) &= 0\ \hspace*{2.5cm}  \mathrm{on}\ \partial \Omega,
			\end{array}
			\right.
		\end{align}
have different Morse indices for $\lambda=0,1$, then there is a bifurcation of \eqref{eq:pdeGeneral2}.
\end{prop}

\begin{proof}
We consider on $H^1_0(\Omega,\mathbb{R}^2)$ the $\mathbb{Z}_2$-action defined by $g\cdot(u,v)=(u,-v)$, where $g$ denotes the non-trivial element in $\mathbb{Z}_2$. Then \eqref{examplefcritpoint} is invariant under this action and by Theorem \ref{thm:main} we need to show $\sfl_G(L)\neq 0\in RO(\mathbb{Z}_2)$. Note that $L_0$, $L_1$ are invertible by the assumption on \eqref{eq:pdeGeneral2LIN}. For the spectral flow, we note that by \eqref{phi2} it is enough to show that the spectral flow of the restriction of $L$ to the fixed-point space $H^G=H^1_0(\Omega,\mathbb{R})\oplus\{0\}$ of the action is non-trivial. This restriction is determined by

\[\langle L_\lambda\mid_{H^G} u_1,u_2\rangle=\int_\Omega{\langle\nabla u_1(x),\nabla u_2(x)\rangle\,dx}+\int_\Omega{ a_\lambda(x)u_1(x)u_2(x) \, dx}\]
and we now need to show that the classical spectral flow \eqref{sfl} of $L\mid_{H^G}=\{L_\lambda\mid_{H^G}\}_{\lambda\in I}$ is non-trivial. The operators $L_\lambda\mid_{H^G}$ belong to $\mathcal{FS}^+(H^G)$, and for paths in this component of $\mathcal{FS}(H^G)$ the spectral flow \eqref{sfl} is the difference of the Morse-indices at the endpoints of the path, i.e.

\[\sfl(L\mid_{H^G})=\mu_-(L_0\mid_{H^G})-\mu_-(L_1\mid_{H^G}).\]
This latter fact can be found in \cite[Prop. 3.9]{SFLPejsachowiczI} and \cite[Lem. 4.3]{CompSfl}, where it was shown directly from the definition \eqref{sfl}. Finally, the Morse index $\mu_-(L_\lambda\mid_{H^G})$ is the number of negative eigenvalues including multiplicities of \eqref{eq-prop-ExI} (see, e.g., \cite[Lem. 2.5]{AleIchIndef}), which proves the proposition. 
\end{proof}
\noindent
Finally let us have another look at \eqref{eq:pdeExAFi}, where the spectral flow vanishes and there is no bifurcation. The previous proposition shows that even in case of a vanishing spectral flow there can still be bifurcation as long as a symmetry assumption as \eqref{ex-evencond} is required and thus the situation is not necessarily as bad as predicted by Theorem \ref{bif-thm-AFi}. For example, if we modify the function $F$ in \eqref{eq:pdeExAFi} to $F(\lambda,x,u,v)=\frac{\lambda}{2}(u^2-v^2)+R(\lambda,x,u,v)$ for some $R$ that satisfies \eqref{ex-evencond} as well as $D^2_0R(\lambda,x,\cdot,\cdot)=0$ for all $(\lambda,x)$, then the corresponding equation is

\begin{align*}
			\left\{
			\begin{array}{rl}
				-\Delta u(x) &= \lambda u(x)+\nabla_uR(\lambda,x,u(x),v(x))\ \hspace*{0.25cm} \mathrm{in}\ \Omega\\		
				\Delta v(x) &= -\lambda v(x)+\nabla_vR(\lambda,x,u(x),v(x))\ \hspace*{0.25cm} \mathrm{in}\ \Omega\\	
				u(x) &=v(x)= 0\ \hspace*{2.5cm}  \mathrm{on}\ \partial \Omega.
			\end{array}
			\right.
		\end{align*}
Now again $S_\lambda(x)=\diag(\lambda,-\lambda)$ and thus as before $\sfl(L)=0$. However, by Proposition \ref{prop-ExI} there is a bifurcation as long as 

\begin{align*}
			\left\{
			\begin{array}{rl}
				-\Delta u(x) &= \lambda u(x)\ \hspace*{0.25cm} \mathrm{in}\ \Omega\\			
				u(x) &= 0\ \hspace*{2.5cm}  \mathrm{on}\ \partial \Omega
			\end{array}
			\right.
		\end{align*}
has a non trivial solution for some $\lambda$, i.e., as long as $\lambda$ passes a Dirichlet eigenvalue of the domain $\Omega$.

\subsection{Homoclinics of Hamiltonian Systems}
Note that the operators $L_\lambda$ in the previous section are of the type $L_\lambda=T+K_\lambda$ for a fixed $T\in\mathcal{FS}(H)^G$ and compact operators $K_\lambda$. Thus for \eqref{eq:pdeGeneral2} the spectral flow of the corresponding Hessians $L_\lambda$ in \eqref{ex:Hessian} actually only depends on the endpoints $L_0$ and $L_1$ by Proposition \ref{prop-FixedPerturbation}. In particular, if $S_0(x)=S_1(x)$ for all $x\in\Omega$ in \eqref{ex:Hessian}, then $L$ is a closed path and thus $\sfl_G(L)=0$ as we explained below Proposition \ref{prop-FixedPerturbation}. The aim of this section is to construct a $G=\mathbb{Z}_2$-invariant family of functionals $f$ such that the Hessians $L$ are a loop in $\mathcal{FS}(H)^G$ having a non-vanishing $G$-equivariant spectral flow. Thus by Theorem \ref{thm:main} there is a bifurcation of critical points that could not have been found by any invariant that only depends on the endpoints $L_0$, $L_1$ of the path such as \cite{shortdegree}, \cite{BlaGoRy} and \cite{GoleRy}. Moreover, our example also has the feature that $\sfl(L)=0$ and consequently Theorem \ref{thm-FPR} fails as well.\\
Let $\mathcal{H}:I\times\mathbb{R}\times\mathbb{R}^{2n}\rightarrow\mathbb{R}$ be a smooth map and consider the Hamiltonian systems

\begin{equation}\label{Hamiltoniannonlin}
\left\{
\begin{aligned}
Ju'(t)+\nabla_u \mathcal{H}_\lambda(t,u(t))&=0,\quad t\in\mathbb{R}\\
\lim_{t\rightarrow\pm\infty}u(t)&=0,
\end{aligned}
\right.
\end{equation}
where $\lambda\in I$ and

\begin{align}\label{J}
J=\begin{pmatrix}
0&-I_n\\
I_n&0
\end{pmatrix}
\end{align}
is the standard symplectic matrix. In what follows, we assume that $\mathcal{H}$ is of the form

\begin{align}\label{Hgrowth}
\mathcal{H}_\lambda(t,u)=\frac{1}{2}\langle A(\lambda,t)u,u\rangle+R(\lambda,t,u),
\end{align}
where $A:I\times\mathbb{R}\rightarrow\mathcal{L}(\mathbb{R}^{2n})$ is a family of symmetric matrices, $R(\lambda,t,u)$ vanishes up to second order at $u=0$, and there are $p>0$, $C\geq 0$ and $r\in H^1(\mathbb{R},\mathbb{R})$ such that

\[|D^2_uR(\lambda,t,u)|\leq r(t)+C|u|^p.\]
Moreover, we suppose that $A_\lambda:=A(\lambda,\cdot):\mathbb{R}\rightarrow\mathcal{L}(\mathbb{R}^{2n})$ converges uniformly in $\lambda$ to families

\begin{align}\label{limits}
A_\lambda(+\infty):=\lim_{t\rightarrow\infty}A_\lambda(t),\quad A_\lambda(-\infty):=\lim_{t\rightarrow-\infty}A_\lambda(t),\quad\lambda\in I,
\end{align}
and that the matrices $JA_\lambda(\pm\infty)$ are hyperbolic, i.e. they have no eigenvalues on the imaginary axis. Note that by \eqref{Hgrowth}, $\nabla_u \mathcal{H}_\lambda(t,0)=0$ for all $(\lambda,t)\in I\times\mathbb{R}$, so that $u\equiv 0$ is a solution of \eqref{Hamiltoniannonlin} for all $\lambda\in I$.\\
Let us now briefly recall the variational formulation of the equations \eqref{Hamiltoniannonlin} from \cite[\S 4]{Jacobo}. The bilinear form $b(u,v)=\langle J u',v\rangle_{L^2(\mathbb{R},\mathbb{R}^{2n})}$, $u,v\in H^1(\mathbb{R},\mathbb{R}^{2n})$, extends to a bounded form on the well known fractional Sobolev space $H^\frac{1}{2}(\mathbb{R},\mathbb{R}^{2n})$. Under the assumption \eqref{Hgrowth}, the map $f:I\times H^\frac{1}{2}(\mathbb{R},\mathbb{R}^{2n})\rightarrow\mathbb{R}$ given by

\[f_\lambda:H^\frac{1}{2}(\mathbb{R},\mathbb{R}^{2n})\rightarrow\mathbb{R},\quad f_\lambda(u)=\frac{1}{2}b(u,u)+\frac{1}{2}\int^\infty_{-\infty}{\langle A(\lambda,t)u(t),u(t)\rangle\,dt}+\int^\infty_{-\infty}{R(\lambda,t,u(t))\,dt}\]
is $C^2$. Moreover, it was shown in \cite{Jacobo} that its critical points are the classical solutions of \eqref{Hamiltoniannonlin} and each sequence of critical points that converges to a bifurcation point actually converges in $C^1(\mathbb{R},\mathbb{R}^{2n})$. Finally,
the second derivative of $f_\lambda$ at the critical point $0\in H^\frac{1}{2}(\mathbb{R},\mathbb{R}^{2n})$ is given by

\begin{align}\label{L}
D^2_0f_\lambda(u,v)=b(u,v)+\int^\infty_{-\infty}{\langle A(\lambda,t)u(t),v(t)\rangle\,dt}
\end{align}
and, by using the hyperbolicity of $J A_\lambda(\pm\infty)$, it can be shown that the corresponding Riesz representations $L_\lambda:H^\frac{1}{2}(\mathbb{R},\mathbb{R}^{2n})\rightarrow H^\frac{1}{2}(\mathbb{R},\mathbb{R}^{2n})$ are Fredholm. Consequently, the operators $L_\lambda$ are selfadjoint Fredholm operators, and it follows by elliptic regularity that the kernel of $L_\lambda$ consists of the classical solutions of the linear differential equation

\begin{equation}\label{Hamiltonianlin}
\left\{
\begin{aligned}
Ju'(t)+A(\lambda,t)u(t)&=0,\quad t\in\mathbb{R}\\
\lim_{t\rightarrow\pm\infty}u(t)&=0.
\end{aligned}
\right.
\end{equation}
The stable and the unstable subspaces of \eqref{Hamiltonianlin} are

\begin{align*}
E^s(\lambda,0)&=\{u(0)\in\mathbb{R}^{2n}:\,Ju'(t)+A(\lambda,t)u(t)=0,\, t\in\mathbb{R}; u(t)\rightarrow 0, t\rightarrow\infty\},\\
E^u(\lambda,0)&=\{u(0)\in\mathbb{R}^{2n}:\,Ju'(t)+A(\lambda,t)u(t)=0,\, t\in\mathbb{R}; u(t)\rightarrow 0, t\rightarrow-\infty\},
\end{align*}
and it is clear that \eqref{Hamiltonianlin} has a non-trivial solution if and only if $E^s(\lambda,0)$ and $E^u(\lambda,0)$ intersect non-trivially.\\
%
Denote by $g$ the non-trivial element of $G = \mathbb{Z}_2$. We set

\begin{align}
\rho(g)=\begin{pmatrix}
1&0&0&0\\
0&-1&0&0\\
0&0&1&0\\
0&0&0&-1
\end{pmatrix}
\end{align}
and consider Hamitonian systems in $\mathbb{R}^{4}$ (c.f. \cite{ArioliSzulkin}, \cite{Clapp}), where

\begin{align}\label{HamMatrix}
A(\lambda,t)=\begin{pmatrix}
a_\lambda(t)&0&c_\lambda(t)&0\\
0&b_\lambda(t)&0&d_\lambda(t)\\
c_\lambda(t)&0&e_\lambda(t)&0\\
0&d_\lambda(t)&0&h_\lambda(t)
\end{pmatrix}
\end{align} 
is equivariant under the action of $G$ for any functions $a,b,c,d,e,h:I\times\mathbb{R}\rightarrow\mathbb{R}$. Now the fixed point space of our action is

\[H^G=\{(u_1,u_2,u_3,u_4)\in H^\frac{1}{2}(\mathbb{R},\mathbb{R}^{4}):\, u_2=u_4=0\}\]
and it follows from \eqref{L} that the kernel of $L_\lambda\mid_{H^G}$ is made of the solutions of the Hamiltonian systems

\begin{equation}\label{Hamiltonianlinreduced}
\left\{
\begin{aligned}
J\begin{pmatrix}
u'_1\\u'_3
\end{pmatrix}+\begin{pmatrix}
a_\lambda(t)&c_\lambda(t)\\
c_\lambda(t)&e_\lambda(t)
\end{pmatrix}
\begin{pmatrix}
u_1\\u_3
\end{pmatrix}&=0,\quad t\in\mathbb{R}\\
\lim_{t\rightarrow\pm\infty}u(t)&=0,
\end{aligned}
\right.
\end{equation}
in $\mathbb{R}^2$, and likewise the kernel of $L_\lambda\mid_{(H^G)^\perp}$ consists of the solutions of

\begin{equation}\label{HamiltonianlinreducedII}
\left\{
\begin{aligned}
J\begin{pmatrix}
u'_2\\u'_4
\end{pmatrix}+\begin{pmatrix}
b_\lambda(t)&d_\lambda(t)\\
d_\lambda(t)&h_\lambda(t)
\end{pmatrix}
\begin{pmatrix}
u_2\\u_4
\end{pmatrix}&=0,\quad t\in\mathbb{R}\\
\lim_{t\rightarrow\pm\infty}u(t)&=0.
\end{aligned}
\right.
\end{equation}
We now use an example of Pejsachowicz from \cite{Jacobo} to construct a loop of operators $L=\{L_\lambda\}_{\lambda\in I}$ such that $\sfl(L)=0$ but $\sfl_G(L)\in RO(\mathbb{Z}_2)$ is non-trivial. To keep our formulas as simple as possible, we use instead of $I=[0,1]$ as parameter interval $[-\pi,\pi]$ and consider for $\lambda\in[-\pi,\pi]$ the matrix family

\begin{align}\label{Az}
\widetilde{A}(\lambda,t)=\begin{pmatrix}
a_\lambda(t)&c_\lambda(t)\\
c_\lambda(t)&e_\lambda(t)
\end{pmatrix}=\begin{cases}
(\arctan t)JS_{\lambda},\quad t\geq 0\\
(\arctan t) JS_0,\quad t<0,
\end{cases},
\end{align}
where

\[S_{\lambda}=\begin{pmatrix}
\cos(\lambda)&\sin(\lambda)\\
\sin(\lambda)&-\cos(\lambda)
\end{pmatrix}.\]
Note that $\tilde{A}(-\pi,t)=\tilde{A}(\pi,t)$ for all $t\in\mathbb{R}$.\\
The space $\mathbb{R}^2$ is symplectic with respect to the canonical symplectic form $\omega_0(u,v)=\langle Ju,v\rangle_{\mathbb{R}^{2}}$. As the matrices \eqref{Az} converge uniformly in $\lambda$ to families of hyperbolic matrices for $t\rightarrow\pm\infty$, it can be shown that the stable and unstable spaces $E^s(\lambda,0)$, $E^u(\lambda,0)$ are Lagrangian subspaces of $\mathbb{R}^2$ (cf. e.g. \cite[Lemma 4.1]{Homoclinics}). This implies in particular that $E^s(\lambda,0)$ and $E^u(\lambda,0)$ are one-dimensional for all $\lambda\in[-\pi,\pi]$.\\
To find non-trivial solutions of \eqref{Hamiltonianlinreduced}, we now consider $E^u(\lambda,0)\cap E^s(\lambda,0)\neq \{0\}$. By a direct computation it can be checked that

\begin{align*}
u_-(t)&=\sqrt{t^2+1}\,e^{-t\arctan(t)}\begin{pmatrix} 1\\0\end{pmatrix},\, t\leq 0,\\ u_+(t)&=\sqrt{t^2+1}\,e^{-t\arctan(t)}\begin{pmatrix} \cos\left(\frac{\lambda}{2}\right)\\ \sin\left(\frac{\lambda}{2}\right)\end{pmatrix},\,t\geq 0,
\end{align*}
are solutions of \eqref{Hamiltonianlinreduced} on the negative and positive half-line, respectively. As they extend to global solutions and since $t\arctan(t)\rightarrow \infty$ as $t\rightarrow\pm\infty$, we see that $u_-(0)\in E^u(\lambda,0)$ and $u_+(0)\in E^s(\lambda,0)$. As $u_+(0)$ and $u_-(0)$ are linearly dependent if and only if $\lambda=0$, we conclude that \eqref{Hamiltonianlinreduced} has a non-trivial solution if and only if $\lambda=0$, and the kernel of $L_0\mid_{H^G}$ is the span of

\[u_{\ast}(t)=\sqrt{t^2+1}\,e^{-t\arctan(t)}\begin{pmatrix} 1\\0\end{pmatrix},\quad t\in\mathbb{R}.\] 
Next we compute the spectral flow of $L\mid_{H^G}$ by a crossing form \eqref{crossing-form}. We need to consider

\[\Gamma(L\mid_{H^G},0)[u_\ast]=\int^\infty_{-\infty}{\left\langle\dot{\widetilde{A}}(0,t)u_\ast(t),u_{\ast}(t)\right\rangle\,dt},\]
where

\begin{align*}
\dot{\widetilde{A}}(0,t)=\begin{cases}
(\arctan t)J\dot{S}_{0},&\quad t\geq 0\\
0,&\quad t<0,
\end{cases}
\end{align*}
and

\[\dot{S}_{0}=\begin{pmatrix}
0&1\\
1&0	
\end{pmatrix}.\]
Consequently,

\begin{align*}
\Gamma(L\mid_{H_G},0)[u_\ast]&=\int^\infty_{0}{\left\langle\dot{\widetilde{A}}(0,t)u_\ast(t),u_{\ast}(t)\right\rangle\,dt}+\int^0_{-\infty}{\left\langle\dot{\widetilde{A}}(0,t)u_\ast(t),u_{\ast}(t)\right\rangle\,dt}\\
&=\int^\infty_{0}{\arctan(t)\langle J\dot{S}_0 u_{\ast}(t),u_{\ast}(t)\rangle\,dt}\\
&=-\int^\infty_{0}{\arctan(t)(t^2+1)e^{-2t\arctan(t)} \,dt}<0,
\end{align*}
which shows that $\Gamma(L\mid_{H_G},0)$ is non-degenerate and of signature $-1$ as quadratic form on the one-dimensional kernel of $L_0\mid_{H^G}$. Therefore, by \eqref{crossing-form}, $\sfl(L\mid_{H^G})=-1$ and so $\sfl_G(L)$ is non-trivial in $RO(\mathbb{Z}_2)$ by \eqref{phi2}. Thus for these functions $a,c$ and $e$ there is a bifurcation of critical points of $f$ by Theorem \ref{thm:main}, and consequently also a bifurcation of solutions of \eqref{Hamiltoniannonlin} from the trivial solution. Let us once again point out, that this bifurcation cannot be found by invariants that only depend on the endpoints of the path $L$.\\
Note that we have not yet chosen functions $b,d$ and $h$ in \eqref{HamMatrix}, which we now do in a way such that $\sfl(L)=0\in\mathbb{Z}$ to obtain an example where also Theorem \ref{thm-FPR} is not applicable. Let us firstly point out that it readily follows from \eqref{sfl} that the spectral flow changes its sign if we reverse the orientation of the path of operators. We now set for $t\in\mathbb{R}$ and $\lambda\in[-\pi,\pi]$

\[b_\lambda(t)=a_{-\lambda}(t),\quad h_\lambda(t)=e_{-\lambda}(t),\quad d_\lambda(t)=c_{-\lambda}(t).\]
Then $L_\lambda\mid_{(H^G)^\perp}=L_{-\lambda}\mid_{H^G}$ and thus $\sfl(L\mid_{(H^G)^\perp})=-\sfl(L\mid_{H^G})=1$. It follows from \eqref{phi3} that $\sfl(L)=0$ and so our example has all the required properties.


 \subsubsection*{Acknowledgements}
     
    The authors were supported by the Deutsche Forschungsgemeinschaft (DFG, German Research Foundation) - 459826435. 
    


\vspace*{1.3cm}

\begin{minipage}{1.2\textwidth}
\begin{minipage}{0.4\textwidth}
Marek Izydorek\\
Institute of Applied Mathematics\\
Faculty of Applied Physics and Mathematics\\
Gda\'{n}sk University of Technology\\
Narutowicza 11/12, 80-233 Gda\'{n}sk, Poland\\
marek.izydorek@pg.edu.pl\\\\\\
Joanna Janczewska\\
Institute of Applied Mathematics\\
Faculty of Applied Physics and Mathematics\\
Gda\'{n}sk University of Technology\\
Narutowicza 11/12, 80-233 Gda\'{n}sk, Poland\\
joanna.janczewska@pg.edu.pl

\end{minipage}
\hfill
\begin{minipage}{0.6\textwidth}
Maciej Starostka\\
Martin-Luther-Universit\"at Halle-Wittenberg\\
Naturwissenschaftliche Fakult\"at II\\
Institut f\"ur Mathematik\\
06099 Halle (Saale)\\
Germany\\
maciej.starostka@mathematik.uni-halle.de\\\\
Nils Waterstraat\\
Martin-Luther-Universit\"at Halle-Wittenberg\\
Naturwissenschaftliche Fakult\"at II\\
Institut f\"ur Mathematik\\
06099 Halle (Saale)\\
Germany\\
nils.waterstraat@mathematik.uni-halle.de
\end{minipage}
\end{minipage}

\end{document}